\documentclass[12pt]{article}%
\usepackage{amsfonts}
\usepackage{mitpress}
\usepackage{amsmath}
\usepackage{graphicx}
\usepackage{amssymb}%
\providecommand{\U}[1]{\protect\rule{.1in}{.1in}}
\newtheorem{theorem}{Theorem}

\newtheorem{corollary}[theorem]{Corollary}

\newtheorem{definition}[theorem]{Definition}

\newtheorem{lemma}[theorem]{Lemma}

\newtheorem{proposition}[theorem]{Proposition}
\newtheorem{remark}[theorem]{Remark}

\newenvironment{proof}[1][Proof]{\textbf{#1.} }{\ \rule{0.5em}{0.5em}}
\begin{document}

\title{Regularity of generalized Daubechies wavelets reproducing exponential polynomials}
\author{N. Dyn, O. Kounchev, D. Levin, H. Render}
\maketitle

\begin{abstract}
We investigate non-stationary orthogonal wavelets based on a non-stationary
interpolatory subdivision scheme reproducing a given set of exponentials. The
construction is analogous to the construction of Daubechies wavelets using the
subdivision scheme of Deslauriers-Dubuc. The main result is the smoothness of
these Daubechies type wavelets.

\end{abstract}

\bigskip\textbf{Acknowledgement:} The second named author was sponsored
partially by the Alexander von Humboldt Foundation, and the second and fourth
named authors were sponsored by Project DO--2-275/2008 \textquotedblright
Astroinformatics\textquotedblright\ with Bulgarian NSF.

\section{Introduction}

The main purpose of the present paper is to study a non-stationary
interpolatory subdivision scheme reproducing exponentials and the related
Daubechies type wavelets.

In \cite{deslaurierDubuc1989} G. Deslauriers and S. Dubuc investigated a
subdivision scheme based on polynomial interpolation of odd degree $2n-1.$
They proved the existence of the basic limit function $\Phi^{D_{2n}}$ for the
subdivision scheme and determined its order of regularity. Daubechies wavelets
are closely related to this construction since the autocorrelation function of
the Daubechies scaling function is equal to the basic limit function
$\Phi^{D_{2n}}$ (see \cite{daubechies}, \cite{Shen91}, or Section $5$ below).
Moreover, the regularity of the Daubechies wavelets can be derived from the
regularity of the basic limit function $\Phi^{D_{2n}}.$

The motivation for the present work originated from the attempt to define a
new concept of multivariate subdivision scheme \emph{in the spirit of
Deslauriers and Dubuc,} which is based respectively on multivariate
interpolation. According to the Polyharmonic Paradigm introduced in
\cite{okbook} polyharmonic interpolation on parallel hyperplanes (or on
concentric spheres) provides a proper generalization of the one-dimensional
polynomial interpolation and, respectively, generates a natural multivariate
subdivision scheme. We shall present in a forthcoming paper
\cite{DynKounchevLevinRender2} the \emph{polyharmonic subdivision scheme} as
an immediate generalization of the one-dimensional subdivision scheme of
Dubuc-Deslauriers. This subdivision scheme is \textbf{stationary }and
reproduces polyharmonic functions of fixed order. It is reduced to an infinite
family of one-dimensional inherently \textbf{non-stationary} subdivision
schemes reproducing exponential polynomials of special type. Thus the results
about subdivision schemes for exponentials are the building bricks of the
polyharmonic subdivision schemes discussed in detail in
\cite{DynKounchevLevinRender2}, \cite{kounchevKalagTsvet2010},
\cite{kounchevKalag2010}. 

On the other hand, \emph{ }the study of non-stationary subdivision schemes
reproducing general exponential polynomials were initiated in $2003$ in
\cite{dynlevinexp}, where the main results of Deslauriers and Dubuc were
generalized, in particular, existence of basic limit function and its
regularity were proved.\emph{ }For a given set of numbers $\left\{
\lambda_{j}\right\}  _{j=0}^{n},$ called sometimes frequencies, such a scheme
reproduces the space $V=\operatorname*{span}\left\{  e^{\lambda_{j}t}\right\}
_{j=0}^{n}$, and is characterized by a family of symbols $\{a^{[k]}%
(z)\}_{k\in\mathbb{Z}}$, where $a^{[k]}(z)$ defines the refinement rule at
level $k$. Ch. Micchelli proved in \cite{micchelliWAnonstationary} that the
symbols of these schemes are non-negative on the unit circle, whenever the set
of frequencies $\{\lambda_{j}\}_{j=0}^{n}$ are real and symmetric. Based on
this property of the symbols, he applied in \cite{micchelliWAnonstationary}
the construction of Daubechies to the fixed symbol $a^{[0]}(z),$ corresponding
to a set of exponentials defined by the frequencies $\left\{  \lambda
_{1},\lambda_{2},...,\lambda_{n},-\lambda_{1},-\lambda_{2},...,-\lambda
_{n},0\right\}  .$ However, the so constructed wavelets reproduce only the
function $e^{0}=1.$ In contrast to that, in the present paper we elaborate on
a genuine non-stationary scheme: we apply the same construction to the family
of symbols $\left\{  a^{[k]}(z)\right\}  _{k\in\mathbb{Z}}$ of a
\emph{non-stationaly} scheme reproducing $V$, and show that the resulting
non-stationary wavelets generate a \emph{Multiresolution Analysis} (MRA) which
contains $V$. Let us mention that in \cite{Vonesch} a similar construction is
proposed without the validation of the positivity of the symbols for all $k$,
which is necessary for the construction; see more discussion on the above
references at the end of the paper.

The main result of the present paper states that the order of regularity of
the new Daubechies type  wavelets reproducing $e^{\lambda_{j}x}$ for
$j=1,...,n,$ is at least as large as in the case of classical Daubechies
wavelets reproducing polynomials of degree $\leq n-1$ \textit{provided that}
the filters of the Daubechies scheme are chosen at each level in an
appropriate way. The proof of this result depends on the concept of
asymptotically equivalent subdivision schemes developed in
\cite{dynLevinJMAA95}. These regularity results are important for the
regularity of the multivariate polyharmonic subdivision schemes and the
corresponding multivariate wavelets considered in
\cite{DynKounchevLevinRender2}.

What concerns the computational aspects of the present results, let us mention
that in the case of a general set of frequencies $\left\{  \lambda
_{j}\right\}  _{j=1}^{n}$ it is difficult to find concise expressions for the
symbols of subdivision and respectively for the filters of the wavelets, in
particular for the analogs of the important classical polynomials $Q_{n}$
(appearing in equality (\ref{Qn-1}) below), and the polynomials $b^{[k]}(z)$
(in equality (\ref{abk}) below). However we have to note the remarkable fact
that in the special cases of interest for the multivariate polyharmonic
subdivision scheme the corresponding polynomials can be explicitly
constructed, see \cite{DynKounchevLevinRender2}, \cite{kounchevKalagTsvet2010}%
, \cite{kounchevKalag2010}, \cite{kounchevKalagl2011},
\cite{kounchevKalagl2012}. Also, in \cite{kounchevKalagTsvet2010},
\cite{kounchevKalag2010}, \cite{kounchevKalagl2011}, \cite{kounchevKalagl2012}
the polyharmonic subdivision wavelets of Daubechies type were successfully
applied to problems of Image Processing, where the experiments were carried
out on benchmark images as well as on some astronomical images.

Finally, let us mention that after \cite{deBoorDeVoreRon} non-stationary
Wavelet Analysis reproducing exponentials have appeared in a natural way in
the monograph \cite{okbook} in the context of multivariate polyspline wavelets.

The paper is organized as follows: in Section 2 we shall review basic notions
for non-stationary subdivision schemes. Section 3 is devoted to subdivision
schemes for exponential polynomials and we shall give an improved estimate of
the order of regularity of the basic limit function. In Section 4 we shall
briefly review non-stationary multiresolutional analysis. In Section 5 we
shall give present the construction of the Daubechies scaling function via
subdivision schemes and estimate the order of regularity of the new Daubechies
type wavelets.

Finally, we introduce some notations: $\mathbb{N}_{0}$ denotes the set of all
natural numbers including zero, $\mathbb{Z}$ denotes the set of all integers
and we define the grid $2^{-k}\mathbb{Z=}\left\{  j/2^{k}:j\in\mathbb{Z}%
\right\}  .$ By $C^{\ell}\left(  \mathbb{R}\right)  $ we denote the set of all
functions $f:\mathbb{R}\rightarrow\mathbb{C}$ which are $\ell$-times
continuosly differentiable where $\ell\in\mathbb{N}_{0}.$ The Fourier
transform of an integrable function $f:\mathbb{R}\rightarrow\mathbb{C}$ is
defined by
\[
\widehat{f}\left(  \omega\right)  =\int\limits_{-\infty}^{\infty}f\left(
x\right)  e^{-ix\omega}dx.
\]

\section{Basics in nonstationary subdivision schemes}

In this section we shall briefly recall notations and definitions used in
non-stationary subdivision schemes for functions on the real line. The formal
definition of a subdivision scheme is the following:

\begin{definition}
\label{DSubdivision} A (non-stationary) \textbf{subdivision scheme} $S_{0}$ is
given by a family of sequences $\left(  a_{j}^{\left[  k\right]  }\right)
_{j\in\mathbb{Z}}$ of complex numbers indexed by $k\in\mathbb{N}_{0}$, called
the \emph{masks or rules at level} $k,$ such that $a_{j}^{\left[  k\right]
}\neq0$ only for finitely many $j\in\mathbb{Z}.$ Given a sequence of numbers
$f^{0}\left(  j\right)  ,j\in\mathbb{Z},$ one defines inductively a sequence
of functions $f^{k+1}:2^{-\left(  k+1\right)  }\mathbb{Z}\rightarrow
\mathbb{C}$ by the rule
\begin{equation}
f_{j}^{k+1}:=f^{k+1}\left(  \frac{j}{2^{k+1}}\right)  =\sum_{l\in\mathbb{Z}%
}a_{j-2l}^{\left[  k\right]  }f^{k}\left(  \frac{l}{2^{k}}\right)
\qquad\text{for }j\in\mathbb{Z}. \label{Subdivision}%
\end{equation}
If for each $k\in\mathbb{N}_{0}$ the masks $a_{j}^{\left[  k\right]  }%
,j\in\mathbb{Z}$, are identical the scheme is said to be \textbf{stationary}.
The subdivision scheme is called \textbf{interpolatory} if for all $k\geq0$
and $j\in\mathbb{Z}$ holds
\begin{equation}
f_{2j}^{k+1}=f_{j}^{k}. \label{Subdivision-Interpolatory}%
\end{equation}

\end{definition}

An important tool in subdivision schemes is the \emph{symbol} $a^{\left[
k\right]  }$ of the rule, or mask $a_{j}^{\left[  k\right]  },j\in\mathbb{Z}$,
defined by
\[
a^{\left[  k\right]  }\left(  z\right)  :=\sum_{j\in\mathbb{Z}}a_{j}^{\left[
k\right]  }z^{j}%
\]
which is Laurent polynomial since we assume that the sequence $a_{j}^{\left[
k\right]  },j\in\mathbb{Z},$ has finite support. We shall identify the
subdivision scheme $S_{0}$ by its masks $a_{j}^{\left[  k\right]  }%
,j\in\mathbb{Z},$ $k\in\mathbb{N}_{0}$ or its symbols $a^{\left[  k\right]
},k\in\mathbb{N}_{0}.$ It is clear that a subdivision scheme is stationary if
and only if $a^{[k]}(z)=a(z)$, for all $k\in\mathbb{N}_{0}$. Moreover the
scheme is interpolatory if and only if $a_{2j}^{\left[  k\right]  }%
=\delta_{0,j}$, for all $j\in\mathbb{Z}.$ In terms of the symbol $a^{\left[
k\right]  }$ it is easy to prove that this is equivalent to the identity
\[
a^{[k]}(z)+a^{[k]}(-z)=2,\ \ z\in\mathbb{C}\setminus\{0\}.
\]

\begin{definition}
\label{Dconvergent} Let $\ell\in\mathbb{N}_{0}.$ A subdivision scheme $S_{0}$
is called $C^{\ell}$-\textbf{convergent} if for any bounded initial sequence
$\{f_{j}^{0}\}_{j\in\mathbb{Z}}$ there exists $F\in$ $C^{\ell}\left(
\mathbb{R}\right)  $ such that
\begin{equation}
\lim_{k\rightarrow\infty}\sup_{j\in\mathbb{Z}}\left\{  \left\vert F\left(
j2^{-k}\right)  -f_{j}^{k}\right\vert \right\}  =0.\label{Convergent}%
\end{equation}
The function $F$ is called \textbf{limit function of the subdivision scheme}
for the initial sequence $\{f_{j}^{0}\}_{j\in\mathbb{Z}}.$ The limit function
for the initial data function $f_{j}^{0}=\delta_{0j}$ (here $\delta_{0j}$ is
the Kronecker symbol) is called \textbf{basic limit function} of the scheme
and it is denoted by $\Phi_{0}.$
\end{definition}

A subdivision scheme is called \textbf{convergent} if it is just $C^{0}%
$-convergent. A central problem in the theory of subdivision schemes is to
estimate the order of regularity of a limit function whenever it exists. Let
us remark that for an \emph{interpolatory} scheme the limit function $F$ in
Definition \ref{Dconvergent} is easily computed for dyadic numbers $t=\frac
{j}{2^{k}}$ with $k\geq0$ and $j\in\mathbb{Z}$ through the values $f_{j}^{k}$
and convergence of the scheme asks whether it has a continuous extension over
the whole $\mathbb{R}.$

\begin{remark}
\label{Rem3} It is easy to see that for a convergent subdivision scheme
$S_{0}$ with symbols $a^{[k]}(z):=\sum_{j=-N}^{N}a_{j}^{[k]}z^{j}$ for all
$k\in\mathbb{N}_{0}$, the support of the basic limit function $\Phi_{0}$ is
contained in $[-N,N]$.
\end{remark}

In the non-stationary case the following concept is of importance:

\begin{definition}
Let $S_{0}$ be a subdivision scheme given by the masks $\left(  a_{j}^{\left[
k\right]  }\right)  _{j\in\mathbb{Z}}$. For any natural number $m\in
\mathbb{N}_{0}$ define a new subdivision scheme $S_{m}$ by means of the masks
of level $k:$%
\[
a^{\left[  k\right]  ,m}\left(  z\right)  :=a^{\left[  k+m\right]  }\left(
z\right)  \text{ for all }k\in\mathbb{N}_{0}.
\]

\end{definition}

The following result is proved in \cite{dynlevinActa}\textbf{: }

\begin{theorem}
If $S_{0}$ is a convergent subdivision scheme then $S_{m}$ is convergent for
any $m\in\mathbb{N}_{0}.$

The basic limit function of this scheme is \textbf{denoted} by $\Phi_{m}$.
\end{theorem}

The following result is well-known, see e.g. the proof of Theorem 2.1 in
\cite{CoDy96}, for further results see also \cite{Prot02}.

\begin{proposition}
\label{PropSuff}Let $S_{0}$ be a subdivision scheme with symbols $a^{[k]}(z)$
for $k\in\mathbb{N}_{0}$ such that

(i) $a_{j}^{\left[  k\right]  }=0$ for all $\left\vert j\right\vert \geq N$
and all $k\in\mathbb{N}_{0}$ for some fixed integer $N>0.$

(ii) there is a constant $M>0$ such that $\left\vert a^{\left[  k\right]
}\left(  e^{i\omega}\right)  \right\vert \leq M$ for all $k\in\mathbb{N}_{0}$
and
\begin{equation}
\sum_{k=0}^{\infty}\left\vert \frac{1}{2}a^{\left[  k\right]  }\left(
1\right)  -1\right\vert <\infty\text{.} \label{eqCC}%
\end{equation}

Then the infinite product
\[
\prod\limits_{k=1}^{\infty}\frac{1}{2}a^{\left[  k-1\right]  }\left(
e^{i\omega2^{-k}}\right)
\]
converges uniformly on compact subsets of $\mathbb{R}$.
\end{proposition}

Note that by Remark 3 the basic limit function $\Phi_{0}$ of a convergent
subdivision scheme satisfying (i) in Proposition \ref{PropSuff} has compact
support. Hence the continuous function $\Phi_{0}$ is integrable and square
integrable. It follows that the Fourier transform $\widehat{\Phi_{0}}$ of
$\Phi_{0}$ is well-defined, continuous and square integrable.

\begin{proposition}
\label{PropProd}Let $S_{0}$ be a convergent subdivision scheme with symbols
$a^{[k]}(z)$ for $k\in\mathbb{N}_{0}$ satisfying (i) and (ii) in Proposition
\ref{PropSuff}. Then
\[
\widehat{\Phi_{0}}\left(  \omega\right)  =\prod\limits_{k=1}^{\infty}\frac
{1}{2}a^{\left[  k-1\right]  }\left(  e^{i\omega2^{-k}}\right)  .
\]

\end{proposition}

\begin{definition}
A subdivision schemes $S_{0}$ with masks $a_{j}^{\left[  k\right]  }%
,j\in\mathbb{Z}$, $k\in\mathbb{N}_{0}$ \textbf{reproduces a continuous
function} $f:\mathbb{R}\rightarrow\mathbb{C}$ at level $k\in\mathbb{N}_{0}$ if
for all $j\in\mathbb{Z}$
\begin{equation}
f\left(  \frac{j}{2^{k+1}}\right)  =\sum_{l\in\mathbb{Z}}a_{j-2l}^{\left[
k\right]  }f\left(  \frac{l}{2^{k}}\right)  . \label{eqReprod}%
\end{equation}
We say that $S_{0}$ \textbf{reproduces} $f$ \textbf{stepwise} if it reproduces
$f$ at each level $k\in\mathbb{N}_{0}.$\textbf{\ }
\end{definition}

In \cite[p. 31, just called reproducing]{dynlevinActa} this is called just
reproducing, and in \cite[here "stepwise"]{DHSS08} is defined the stepwise reproduction.

\begin{proposition}
\label{PropRep}Let $S_{0}$ be a subdivision scheme with symbols $a^{[k]}(z)$
for $k\in\mathbb{N}_{0},$ let $\lambda\in\mathbb{C},$ and $g\left(  x\right)
$ be a polynomial. Define $z^{\left[  k\right]  }:=\exp\left(  -2^{-(k+1)}%
\lambda\right)  $, and for $\delta=0,1$ define
\[
F_{\delta}^{\left[  k\right]  }\left(  g,x\right)  :=\sum_{l\in\mathbb{Z}%
}a_{\delta-2l}^{\left[  k\right]  }g\left(  \frac{l+x}{2^{k}}\right)  \left(
z^{\left[  k\right]  }\right)  ^{\delta-2l}%
\]
for $\delta=0,1.$

Then for fixed $k\in\mathbb{N}_{0}$ the following statements are equivalent :

a) $g\left(  x\right)  e^{\lambda x}$ is reproduced at level $k.$

b) $F_{\delta}^{\left[  k\right]  }\left(  g,m\right)  =g\left(
\frac{2m+\delta}{2^{k+1}}\right)  $ for $\delta=0,1$ and for all
$m\in\mathbb{Z}$.

c) $F_{\delta}^{\left[  k\right]  }\left(  g,x\right)  =g\left(
\frac{2x+\delta}{2^{k+1}}\right)  $ for $\delta=0,1$ and for all
$x\in\mathbb{R}$.
\end{proposition}

\begin{proof}
Using (\ref{eqReprod}) for $j=2m+\delta$ with $m\in\mathbb{Z}$ and
$\delta=0,1,$ it is easy to see that the function $g\left(  x\right)
e^{\lambda x}$ is reproduced stepwise if and only if for $\delta=0,1$ and for
all $m\in\mathbb{Z}$
\[
g\left(  \frac{2m+\delta}{2^{k+1}}\right)  =\sum_{l\in\mathbb{Z}}%
a_{2m+\delta-2l}^{\left[  k\right]  }g\left(  \frac{l}{2^{k}}\right)  \left(
z^{\left[  k\right]  }\right)  ^{2m+\delta-2l}.
\]
Using the variable transformation $\widetilde{l}=l+m$ it is easy to see that
the right hand side is equal to $F_{\delta}\left(  g,m\right)  ,$ hence the
equivalence of $a)$ and $b)$ is proven. The equivalence of $b)$ and $c)$
follows from the fact that $g\left(  x\right)  $ and $F_{\delta}\left(
g,x\right)  $ are polynomials.
\end{proof}

Theorem \ref{ThmZero} below extends Theorem 2.3 in \cite{dynlevinexp} which
was formulated only for interpolatory schemes. (We will apply this later to
Daubechies schemes which are not interpolatory)\textbf{.}

\begin{theorem}
\label{ThmZero}Let $S_{0}$ be a subdivision scheme with masks $a_{j}^{\left[
k\right]  },j\in\mathbb{Z}$, $k\in\mathbb{N}_{0}$. Then for $r=0,..,\mu-1$ the
functions $f_{r}\left(  x\right)  =x^{r}e^{\lambda x}$ are reproduced stepwise
by $S_{0}$ if and only if for all $k\in\mathbb{N}_{0}$ holds
\begin{equation}
a^{[k]}\left(  -\exp\left(  -2^{-(k+1)}\lambda\right)  \right)  =0\text{ and
}a^{[k]}\left(  \exp\left(  -2^{-(k+1)}\lambda\right)  \right)  =2
\label{eqaaa1}%
\end{equation}
and
\begin{equation}
{\frac{d^{r}}{dz^{r}}}a^{[k]}\left(  \pm\exp\left(  -2^{-(k+1)}\lambda\right)
\right)  =0,\ \ r=1,...,\mu-1. \label{eqaaa2}%
\end{equation}

\end{theorem}

\begin{proof}
Put $z^{\left[  k\right]  }:=\exp\left(  -2^{-(k+1)}\lambda\right)  .$ Since
$F_{\delta}^{\left[  k\right]  }\left(  g_{\delta},0\right)  =\sum
_{l\in\mathbb{Z}}a_{\delta-2l}^{\left[  k\right]  }g_{\delta}\left(  \frac
{l}{2^{k}}\right)  \left(  z^{\left[  k\right]  }\right)  ^{\delta-2l},$ we
first observe the following identities for the constant function $g=1:$
\begin{align}
a^{\left[  k\right]  }\left(  z^{\left[  k\right]  }\right)   &
=F_{0}^{\left[  k\right]  }\left(  1,0\right)  +F_{1}^{\left[  k\right]
}\left(  1,0\right)  ,\label{eqId1}\\
a^{\left[  k\right]  }\left(  -z^{\left[  k\right]  }\right)   &
=F_{0}^{\left[  k\right]  }\left(  1,0\right)  -F_{1}^{\left[  k\right]
}\left(  1,0\right)  . \label{eqId2}%
\end{align}
Suppose $S_{0}$ reproduces the function $f\left(  x\right)  =g\left(
x\right)  e^{\lambda x}$ where $g\left(  x\right)  $ is a polynomial of degree
$\leq\mu-1.$ We obtain from condition $b)$ in Proposition \ref{PropRep} for
$g=1$ that $1=F_{\delta}^{\left[  k\right]  }\left(  1,0\right)  $ for
$\delta=0,1,$ and we conclude from (\ref{eqId1}) and (\ref{eqId2}) that
$a^{\left[  k\right]  }\left(  z^{\left[  k\right]  }\right)  =2$ and
$a^{\left[  k\right]  }\left(  -z^{\left[  k\right]  }\right)  =0.$ Next we
take%
\begin{equation}
g_{\delta}^{s}\left(  x\right)  =\left(  \delta-2^{k+1}x\right)  \left(
\delta-2^{k+1}x-1\right)  ...\left(  \delta-2^{k+1}x-\left(  s-1\right)
\right)  \label{defgg}%
\end{equation}
for $s=1,...,\mu-1.$ Then%
\begin{equation}
g_{\delta}^{s}\left(  \frac{l}{2^{k}}\right)  =\left(  \delta-2l\right)
\left(  \delta-2l-1\right)  ...\left(  \delta-2l-\left(  s-1\right)  \right)
. \label{eq+++}%
\end{equation}
Since $F_{\delta}^{\left[  k\right]  }\left(  g_{\delta},0\right)  =\sum
_{l\in\mathbb{Z}}a_{\delta-2l}^{\left[  k\right]  }g_{\delta}\left(  \frac
{l}{2^{k}}\right)  \left(  z^{\left[  k\right]  }\right)  ^{\delta-2l}$ it
follows from (\ref{eq+++}) that
\begin{align}
\frac{d^{s}}{dz^{s}}a^{\left[  k\right]  }\left(  z^{\left[  k\right]
}\right)   &  =F_{0}^{\left[  k\right]  }\left(  g_{0}^{s},0\right)
+F_{1}^{\left[  k\right]  }\left(  g_{1}^{s},0\right)  ,\label{eqR1}\\
\frac{d^{s}}{dz^{s}}a^{\left[  k\right]  }\left(  -z^{\left[  k\right]
}\right)   &  =F_{0}^{\left[  k\right]  }\left(  g_{0}^{s},0\right)
-F_{1}^{\left[  k\right]  }\left(  g_{1}^{s},0\right)  . \label{eqR2}%
\end{align}
Since $g_{\delta}^{s}\left(  x\right)  e^{\lambda x}$ is reproduced, we obtain
that $F_{\delta}^{\left[  k\right]  }\left(  g_{\delta}^{s},0\right)
=g_{\delta}\left(  \frac{\delta}{2^{k+1}}\right)  =0$ for $\delta=0,1.$ Thus
(\ref{eqaaa2}) holds and the necessity part is proved.

The converse is proved by induction over the dimension $s$ of the linear space
$E_{s}$ of all polynomials $g$ of degree $\leq s.$ By Proposition
\ref{PropRep} we have to prove that for each $g\in E_{s}$ holds
\[
g\left(  \frac{2m+\delta}{2^{k+1}}\right)  =F_{\delta}^{\left[  k\right]
}\left(  g,m\right)  \text{ for }\delta=0,1\text{ and for all }m\in
\mathbb{Z}.
\]
For $s=0$ this means we have to prove that $F_{\delta}^{\left[  k\right]
}\left(  1,m\right)  =1$ for the constant function $g=1,$ for $\delta=0,1,$
and for all $m\in\mathbb{Z}.$ Since $F_{\delta}^{\left[  k\right]  }\left(
g,x\right)  $ is the constant polynomial we have only to show that $F_{\delta
}^{\left[  k\right]  }\left(  1,0\right)  =1.$ By assumption (\ref{eqaaa2})
and equations (\ref{eqId1}) and (\ref{eqId2}) we infer $2=a^{\left[  k\right]
}\left(  z^{\left[  k\right]  }\right)  =F_{0}^{\left[  k\right]  }\left(
1,0\right)  +F_{1}^{\left[  k\right]  }\left(  1,0\right)  $ and $0=a^{\left[
k\right]  }\left(  z^{\left[  k\right]  }\right)  =F_{0}^{\left[  k\right]
}\left(  1,0\right)  -F_{1}^{\left[  k\right]  }\left(  1,0\right)  ,$ from
which we infer the desired statement $F_{\delta}^{\left[  k\right]  }\left(
1,0\right)  =1$ for $\delta=0,1$.

Suppose that the statement is proven for $E_{s-1}.$ We consider the polynomial
$g_{\delta}^{s}$ defined in (\ref{defgg}) for $\delta=0,1.$ Note that equation
(\ref{eqR1}) and (\ref{eqR2}), and the assumption (\ref{eqaaa2}) imply that
\begin{equation}
F_{\delta}^{\left[  k\right]  }\left(  g_{\delta}^{s},0\right)  =0=g_{\delta
}^{s}\left(  \frac{\delta}{2^{k+1}}\right)  \label{eqfnull}%
\end{equation}
for all $s=1,...,\mu-1.$ Put now $g^{s}:=g_{0}^{\delta}.$ Then by
(\ref{eqfnull}) it follows $F_{0}^{\left[  k\right]  }\left(  g^{s},0\right)
=g^{s}\left(  0\right)  $ . Note that
\[
g_{1}^{s}\left(  x\right)  =g_{0}^{s}\left(  x\right)  +G^{s-1}\left(
x\right)
\]
for some polynomial $G^{s-1}$ of degree $\leq s-1.$ Then
\[
F_{1}^{\left[  k\right]  }\left(  g_{0}^{s},0\right)  =F_{1}^{\left[
k\right]  }\left(  g_{1}^{s},0\right)  -F_{1}^{\left[  k\right]  }\left(
G^{s-1},0\right)  =-G^{s-1}\left(  \frac{1}{2^{k+1}}\right)  =g_{0}^{s}\left(
\frac{1}{2^{k+1}}\right)  ,
\]
since $F_{1}^{\left[  k\right]  }\left(  g_{1}^{s},0\right)  =0$ by
(\ref{eqfnull}), and $F_{1}^{\left[  k\right]  }\left(  G^{s-1},0\right)
=G^{s-1}\left(  \frac{1}{2^{k+1}}\right)  $ by the induction hypothesis.
Hence
\[
F_{\delta}^{\left[  k\right]  }\left(  g^{s},0\right)  =g^{s}\left(
\frac{\delta}{2^{k+1}}\right)  .
\]
If we prove that $F_{\delta}^{\left[  k\right]  }\left(  g^{s},m\right)
=g^{s}\left(  \frac{2m+\delta}{2^{k+1}}\right)  $ for all $m\in\mathbb{Z}$ and
$\delta=0,1,$ then we can conclude that $g^{s}$ is reproduced by the scheme
and clearly this implies that $E_{s}$ is reproduced. Note that for each
$m\in\mathbb{Z}$ there exists a polynomial $h_{m}^{s}$ of degree $<s$ such
that
\[
g^{s}\left(  x+\frac{m}{2^{k}}\right)  =g^{s}\left(  x\right)  +h_{m}%
^{s}\left(  x\right)  .
\]
Then $g^{s}\left(  \frac{l}{2^{k}}+\frac{m}{2^{k}}\right)  =g^{s}\left(
\frac{l}{2^{k}}\right)  +h_{m}^{s}\left(  \frac{l}{2^{k}}\right)  $ and
\[
F_{\delta}^{\left[  k\right]  }\left(  g^{s},m\right)  =\sum_{l\in\mathbb{Z}%
}a_{\delta-2l}^{\left[  k\right]  }g^{s}\left(  \frac{l+m}{2^{k}}\right)
\left(  z^{\left[  k\right]  }\right)  ^{\delta-2l}=F_{\delta}^{\left[
k\right]  }\left(  g^{s},0\right)  +F_{\delta}^{\left[  k\right]  }\left(
h_{m}^{s},0\right)  .
\]
By the induction hypothesis we know that $F_{\delta}^{\left[  k\right]
}\left(  h_{m}^{s},0\right)  =h_{m}^{s}\left(  \frac{\delta}{2^{k+1}}\right)
.$ It follows from (\ref{eqfnull}) that
\[
F_{\delta}^{\left[  k\right]  }\left(  g^{s},m\right)  =g^{s}\left(
\frac{\delta}{2^{k+1}}\right)  +h_{m}^{s}\left(  \frac{\delta}{2^{k+1}%
}\right)  =g^{s}\left(  \frac{\delta+2m}{2^{k+1}}\right)  .
\]

The proof is complete.
\end{proof}

It is easy to see that a function $f:\mathbb{R}\rightarrow\mathbb{C}$ is
reproduced stepwise by a subdivision scheme if and only if for the data
function $f^{\left[  0\right]  }\left(  j\right)  :=f\left(  j\right)  $ one
has
\[
f^{\left[  k\right]  }\left(  \frac{j}{2^{k}}\right)  =f\left(  \frac{j}%
{2^{k}}\right)  .
\]

From this it is easy to see that a convergent and stepwise reproducing
subdivision scheme is reproducing in the following sense:

\begin{definition}
A convergent subdivision scheme $S_{0}$ with masks $a_{j}^{\left[  k\right]
},j\in\mathbb{Z}$, $k\in\mathbb{N}_{0},$ is reproducing a continuous function
$f:\mathbb{R}\rightarrow\mathbb{C}$ if the limit function for the data
function $f\left(  j\right)  ,$ $j\in\mathbb{Z},$ is equal to $f\left(
x\right)  .$
\end{definition}

\begin{remark}
In the definition of a convergent scheme some authors require bounded data.
\end{remark}

An important concept for the investigation of non-stationary subdivision
schemes is the following notion introduced in \cite{dynLevinJMAA95}:

\begin{definition}
Two subdivision schemes $S_{a}$ and $S_{b}$ with masks $a_{j}^{\left[
k\right]  },j\in\mathbb{Z}$, $k\in\mathbb{N}_{0},$ and $b_{j}^{\left[
k\right]  },j\in\mathbb{Z}$, $k\in\mathbb{N}_{0},$resp. are called
\textbf{asymptotically equivalent} if%
\begin{equation}
\sum_{k=0}^{\infty}\sum_{j\in\mathbb{Z}}\left\vert a_{j}^{\left[  k\right]
}-b_{j}^{\left[  k\right]  }\right\vert <\infty. \label{eqAE}%
\end{equation}
We say that $S_{a}$ and $S_{b}$ are \textbf{exponentially asymptotically
equivalent} if there exists a constant $C>0$ such that
\begin{equation}
\max_{j\in\mathbb{Z}}\left\vert a_{j}^{\left[  k\right]  }-b_{j}^{\left[
k\right]  }\right\vert \leq C\cdot2^{-k} \label{eqAE2}%
\end{equation}
for all $k\in\mathbb{N}_{0}.$
\end{definition}

We also say that the masks $a^{\left[  k\right]  }$, $k\in\mathbb{N}_{0},$ and
$b^{\left[  k\right]  },$ $k\in\mathbb{N}_{0},$ are (exponentially resp.)
asymptotically equivalent if (\ref{eqAE}) ((\ref{eqAE2}) resp.) holds.

Suppose that the masks $a_{j}^{\left[  k\right]  },j\in\mathbb{Z}$,
$k\in\mathbb{N}_{0},$ and $b_{j}^{\left[  k\right]  },j\in\mathbb{Z}$,
$k\in\mathbb{N}_{0},$ have support in the set $\left\{  -N,...,N\right\}  ,$
i.e. $a_{j}^{\left[  k\right]  }=0$ for $\left\vert j\right\vert >N.$ Then it
is easy to see that $a^{\left[  k\right]  }$, $k\in\mathbb{N}_{0},$ and
$b^{\left[  k\right]  },$ $k\in\mathbb{N}_{0},$ are exponentially
asymptotically equivalent if and only if for any $R>1$ there exists $D>0$ such
that
\[
\left\vert a^{\left[  k\right]  }\left(  z\right)  -b^{\left[  k\right]
}\left(  z\right)  \right\vert \leq D\cdot2^{-k}%
\]
for all $k\in\mathbb{N}_{0}$ and for all $z\in\mathbb{C}$ with $1/R\leq
\left\vert z\right\vert \leq R.$

The following result provides a sufficient method for constructing
asymptotically equivalent subdivision schemes.

\begin{theorem}
\label{ThmSuff}Let $m\in\mathbb{N}_{0}$ and assume that $p^{\left[  k\right]
}\left(  z\right)  $ and $p\left(  z\right)  $ are polynomials of degree $m$
for each $k\in\mathbb{N}_{0}$ defined as
\[
p^{\left[  k\right]  }\left(  z\right)  =c^{\left[  k\right]  }\prod
\limits_{j=1}^{m}\left(  z-\alpha_{j}^{\left[  k\right]  }\right)  \text{ and
}p\left(  z\right)  =c\prod\limits_{j=1}^{m}\left(  z-\alpha_{j}\right)
\]
for some complex numbers $c^{\left[  k\right]  },c,$ and $\alpha_{j}^{\left[
k\right]  }$ and $\alpha_{j},$ for $j=1...,m,$ and $k\in\mathbb{N}_{0}.$
Suppose that there exists a constant $D_{m}>0$ such that for all
$k\in\mathbb{N}_{0}$ and $j=1,...,m$
\begin{equation}
\left\vert \alpha_{j}^{\left[  k\right]  }-\alpha_{j}\right\vert \leq
D_{m}2^{-k}\text{ and }\left\vert c^{\left[  k\right]  }-c\right\vert \leq
D_{m}2^{-k}. \label{eqexp}%
\end{equation}
Then $p^{\left[  k\right]  }\left(  z\right)  ,k\in\mathbb{N}_{0}$, and
$p\left(  z\right)  ,$ $k\in\mathbb{N}_{0},$ are exponentially asymptotically equivalent.
\end{theorem}

\begin{proof}
We claim by induction over $m\in\mathbb{N}_{0}$ that for each $R>0$ there
exists a constant $C_{m}\left(  R\right)  $ such that for all $k\in
\mathbb{N}_{0}$ and $\left\vert z\right\vert \leq R$
\begin{equation}
\left\vert p^{\left[  k\right]  }\left(  z\right)  -p\left(  z\right)
\right\vert \leq C_{m}\left(  r\right)  2^{-k}. \label{eqpexp}%
\end{equation}
For $m=0$ the statement follows directly from (\ref{eqexp}). Suppose that the
statement is true for $m-1.$ Write%
\[
p^{\left[  k\right]  }\left(  z\right)  =c^{\left[  k\right]  }\left(
z-\alpha_{m}^{\left[  k\right]  }\right)  p_{m-1}^{\left[  k\right]  }\left(
z\right)  \text{ and }p\left(  z\right)  =c\left(  z-\alpha_{m}\right)
p_{m-1}\left(  z\right)  ,
\]
where $p_{m-1}^{\left[  k\right]  }\left(  z\right)  $ and $p_{m-1}\left(
z\right)  $ are polynomials of degree $\leq m-1$ and leading coefficient $1.$
By the induction hypothesis, for each $R>0$ a constant $C_{m}\left(  R\right)
$ such that $\left\vert c^{\left[  k\right]  }p_{m-1}^{\left[  k\right]
}\left(  z\right)  -cp_{m-1}\left(  z\right)  \right\vert \leq C_{m}\left(
r\right)  2^{-k}$ for all $k\in\mathbb{N}_{0}$ and $\left\vert z\right\vert
\leq R.$ Note that
\begin{align*}
p^{\left[  k\right]  }\left(  z\right)  -p\left(  z\right)   &  =c^{\left[
k\right]  }\left(  z-\alpha_{m}^{\left[  k\right]  }\right)  p_{m-1}^{\left[
k\right]  }\left(  z\right)  -c\left(  z-\alpha_{m}\right)  p_{m-1}\left(
z\right) \\
&  =\left(  z-\alpha_{m}\right)  \left[  c^{\left[  k\right]  }p_{m-1}%
^{\left[  k\right]  }\left(  z\right)  -cp_{m-1}\right]  +c^{\left[  k\right]
}\left(  \alpha_{m}-\alpha_{m}^{\left[  k\right]  }\right)  p_{m-1}^{\left[
k\right]  }\left(  z\right)  .
\end{align*}
Using the triangle inequality, the induction hypothesis, and (\ref{eqexp}),
one obtains (\ref{eqpexp}).
\end{proof}

\begin{theorem}
\label{ThmNec}Assume that $p^{\left[  k\right]  }\left(  z\right)  $ and
$p\left(  z\right)  $ are polynomials of degree $\leq m$ for each
$k\in\mathbb{N}_{0}$ and assume for each $R>0$ there exists a constant
$C_{m}\left(  R\right)  $ such that for all $k\in\mathbb{N}_{0}$ and
$\left\vert z\right\vert \leq R$
\begin{equation}
\left\vert p^{\left[  k\right]  }\left(  z\right)  -p\left(  z\right)
\right\vert \leq C_{m}\left(  R\right)  2^{-k}. \label{eqcmr}%
\end{equation}
If $\alpha$ is a simple zero of $p\left(  z\right)  $ then there exists
$k_{0}\in\mathbb{N}_{0}$ and a constant $\rho>0$ such that for each natural
number $k\geq k_{0}$ there exists a zero $\alpha^{\left[  k\right]  }$ of
$p^{\left[  k\right]  }\left(  z\right)  $ with
\[
\left\vert \alpha^{\left[  k\right]  }-\alpha\right\vert \leq\rho2^{-k}\text{
for all }k\geq k_{0}.
\]

\end{theorem}

\begin{proof}
Let $\rho>0$ and define $\gamma_{k}\left(  t\right)  =\alpha+2^{-k}\rho
e^{it}$ for $t\in\left[  0,2\pi\right]  .$ Write $p\left(  z\right)
=\sum_{j=1}^{m}p_{j}\left(  z-\alpha\right)  ^{j}.$ Since $p\left(  z\right)
$ has a simple zero in $\alpha$ we know that $p_{0}=0$ and $p_{1}\neq0.$ We
obtain the estimate
\[
\left\vert p\left(  \gamma_{k}\left(  t\right)  \right)  \right\vert \geq
\rho2^{-k}\left(  \left\vert p_{1}\right\vert -\sum_{j=2}^{m}\left\vert
p_{j}\right\vert \frac{\rho^{j}}{2^{k\left(  j-1\right)  }}\right)  .
\]
Hence for given $\rho>0$ there exists a natural number $k_{\rho}$ such that
for all $k\geq k_{\rho}$%
\[
\left\vert p\left(  \gamma_{k}\left(  t\right)  \right)  \right\vert \geq
\rho2^{-k}\frac{\left\vert p_{1}\right\vert }{2}%
\]
and for all $t\in\left[  0,2\pi\right]  .$ Now take $R>0$ is large enough that
$R>2\left\vert \alpha\right\vert .$ Take $\rho$ such that $\rho\frac
{\left\vert p_{1}\right\vert }{2}>C_{m}\left(  R\right)  $ where $C_{m}\left(
R\right)  $ is the constant in (\ref{eqcmr}). Take $k_{1}\geq k_{\rho}$ large
enough that $\left\vert \gamma_{k}\left(  t\right)  \right\vert \leq R$ for
all $k\geq k_{1}$ and $t\in\left[  0,2\pi\right]  .$ Then
\[
\left\vert p^{\left[  k\right]  }\left(  \gamma_{k}\left(  t\right)  \right)
-p\left(  \gamma_{k}\left(  t\right)  \right)  \right\vert \leq C_{m}\left(
R\right)  2^{-k}<\rho\frac{\left\vert p_{1}\right\vert }{2}2^{-k}<\left\vert
p\left(  \gamma_{k}\left(  t\right)  \right)  \right\vert
\]
for all $k\geq k_{1}$ and $t\in\left[  0,2\pi\right]  .$ By Rouche's theorem
the number of zeros of $p^{\left[  k\right]  }\left(  z\right)  $ in the ball
$\left\vert z-\alpha\right\vert <2^{-k}\rho$ is equal to the number of zeros
of $p\left(  z\right)  $ in that ball. Since $p\left(  \alpha\right)  =0$ it
follows that for each $k\geq k_{1}$ there exists a zero $\alpha^{\left[
k\right]  }$ of $p^{\left[  k\right]  }\left(  z\right)  $ in the ball
$\left\vert z-\alpha\right\vert <2^{-k}\rho,$ hence, $\left\vert
\alpha^{\left[  k\right]  }-\alpha\right\vert <2^{-k}\rho.$
\end{proof}

\section{Subdivision schemes based on exponential interpolation and regularity
of the basic limit function}

Let us first recall some basic facts about the classical $2n$-point
Deslauriers-Dubuc subdivision scheme.

It is defined via interpolation of polynomials of degree $2n-1$, see
\cite{deslaurierDubuc1989}. We shall denote its symbol by $D_{2n}\left(
z\right)  $ which is given by
\begin{equation}
D_{2n}(z)=\sum_{\left\vert j\right\vert \leq2n-1}p_{j}z^{j}.
\label{Dubuc-D-symbol}%
\end{equation}
According to Theorem 6.1 in \cite{deslaurierDubuc1989} the symbol
$D_{2n}\left(  z\right)  $ has the (reproduction) property if and only if
\begin{equation}
{\frac{d^{j}}{dz^{j}}}D_{2n}(-1)=0,\ \text{for }j=0,...,2n-1.
\label{reconstruction}%
\end{equation}
Further the scheme of Deslauriers and Dubuc is interpolatory, i.e.,
$p_{2j}=\delta_{0,j}$, for all $j\in\mathbb{Z}$, or equivalently
\begin{equation}
D_{2n}(z)+D_{2n}(-z)=2,\ \text{for all }\ z\in\mathbb{C}\setminus\{0\}.
\label{interpolation}%
\end{equation}
Together, conditions (\ref{reconstruction}) and (\ref{interpolation})
constitute a linear system which uniquely determines the symbol $D_{2n}(z)$ of
the Deslauriers and Dubuc scheme and it can be written in the form
\begin{equation}
D_{2n}\left(  z\right)  =\left(  \frac{1+z}{2}\right)  ^{2n}b_{D_{2n}}\left(
z\right)  . \label{ab}%
\end{equation}
Let us mention that condition (\ref{reconstruction}) means that polynomials of
degree $\leq2n-1$ are reproduced by the subdivision scheme. The Laurent
polynomial $b_{D_{2n}}\left(  z\right)  $ can be computed explicitly and the
following identity (see e.g. \cite{DynKounchevLevinRender2}),
\begin{equation}
D_{2n}\left(  z\right)  =2\frac{\left(  1+z\right)  ^{n}\left(  1+\frac{1}%
{z}\right)  ^{n}}{2^{2n}}Q_{n-1}\left(  \varphi\left(  z\right)  \right)
\label{eqD2nB}%
\end{equation}
holds for all $z\neq0,$ where $\varphi\left(  z\right)  =\frac{1}{2}-\frac
{1}{4}\left(  z+1/z\right)  $ and $Q_{n-1}\left(  x\right)  $ is the
polynomial of degree $n-1$ given by
\begin{equation}
Q_{n-1}\left(  x\right)  =\sum_{j=0}^{n-1}\binom{n+j-1}{j}x^{j}=\sum
_{j=0}^{n-1}\frac{\left(  n+j-1\right)  !}{j!\left(  n-1\right)  !}x^{j}.
\label{Qn-1}%
\end{equation}
The next result shows that $Q_{n-1}\left(  z\right)  $ and $b_{D_{2n}}\left(
z\right)  $ have only simple zeros.

\begin{proposition}
\label{PropQQQ}The polynomial $Q_{n-1}\left(  x\right)  $ in (\ref{Qn-1})
satisfies the identity
\[
nQ_{n-1}\left(  x\right)  +\left(  x-1\right)  Q_{n-1}^{\prime}\left(
x\right)  =\frac{\left(  2n-1\right)  !}{\left(  n-1\right)  !\left(
n-1\right)  !}x^{n-1}.
\]

\end{proposition}

\begin{proof}
Clearly $Q_{n-1}^{\prime}\left(  x\right)  =\sum_{j=0}^{n-2}\frac{\left(
n+j\right)  !}{j!\left(  n-1\right)  !}x^{j}.$ Then $nQ_{n-1}\left(  x\right)
+\left(  x-1\right)  Q_{n-1}^{\prime}\left(  x\right)  $ is equal to
\[
n\sum_{j=0}^{n-1}\frac{\left(  n+j-1\right)  !}{j!\left(  n-1\right)  !}%
x^{j}+\sum_{j=1}^{n-1}\frac{\left(  n+j-1\right)  !}{\left(  j-1\right)
!\left(  n-1\right)  !}x^{j}-\sum_{j=0}^{n-2}\frac{\left(  n+j\right)
!}{j!\left(  n-1\right)  !}x^{j}.
\]
from which the statement follows.
\end{proof}

Now we turn to subdivision schemes for exponential polynomials. Let $L$ be the
linear differential operator given by
\begin{equation}
L=\left(  \frac{d}{dx}-\lambda_{0}\right)  ....\left(  \frac{d}{dx}%
-\lambda_{n}\right)  . \label{eqdefL}%
\end{equation}
Complex-valued solutions $f$ of the equation $Lf=0$ are called $L$%
-\emph{polynomials} or \emph{\ exponential polynomials} or just
\emph{exponentials}. We shall denote the set of all solutions of $Lf=0$ by
$E\left(  \lambda_{0}...,\lambda_{n}\right)  .$

Throughout this article we shall assume that $\lambda_{0},...,\lambda_{n}$ are
real numbers. This assumption implies that $E\left(  \lambda_{0}%
...,\lambda_{n}\right)  $ is an \textbf{extended Chebyshev space}, in
particular for any pairwise distinct points $t_{0},...,t_{n}$ and data values
$y_{0},...,y_{n},$ there exists a unique element $p\in E\left(  \lambda
_{0}...,\lambda_{n}\right)  $ with $p\left(  t_{j}\right)  =f_{j}\left(
t_{j}\right)  $ for $j=0,...,n,$ i.e. $p$ is interpolating exponential
polynomial, cf. \cite{kreinNudelman}.

Given real numbers $\lambda_{0},...,\lambda_{2n-1}$ one can define the
\emph{subdivision scheme based on interpolation in }$E\left(  \lambda
_{0}...,\lambda_{2n-1}\right)  $: the new value $f^{k+1}\left(  j/2^{k+1}%
\right)  $ is computed by constructing the unique function $p_{j}\in E\left(
\lambda_{0}...,\lambda_{2n-1}\right)  $ interpolating the previous data
$f^{k}\left(  \left(  j+l\right)  /2^{k}\right)  $ for $l=-n+1,...,n,$ and
putting $f^{k+1}\left(  j/2^{k+1}\right)  =p_{j}\left(  j/2^{k+1}\right)  $
(see \cite{micchelliWAnonstationary} for details). Then the symbols of this
scheme are of the form
\begin{equation}
a^{[k]}(z)=\sum_{|j|\leq2n-1}a_{j}^{[k]}z^{j} \label{DDexp}%
\end{equation}
and since the scheme is interpolatory one has
\begin{equation}
a^{[k]}(z)+a^{[k]}(-z)=2,\ \ z\in\mathbb{C}\setminus\{0\}.
\label{NSinterpolation}%
\end{equation}
Due to the interpolatory definition of the subdivision it is clear that each
function $f\in E\left(  \lambda_{0}...,\lambda_{2n-1}\right)  $ is reproduced
stepwise by the scheme, and Theorem \ref{ThmZero} implies that
\begin{equation}
{\frac{d^{s}}{dz^{s}}}a^{[k]}\left(  -\exp\left(  -2^{-(k+1)}\lambda
_{j}\right)  \right)  =0,\ \ s=0,...,\mu_{j}-1. \label{NSreconstruction2}%
\end{equation}
where $\mu_{j}$ is the \emph{multiplicity} of $\lambda_{j},$ i.e. the number
of times the value $\lambda_{j}$ occurs in $\left(  \lambda_{0},...,\lambda
_{2n-1}\right)  .$ Hence the subdivision scheme based on interpolation in
$E\left(  \lambda_{0},...,\lambda_{2n-1}\right)  $ is completely characterized
by (\ref{NSinterpolation}) and (\ref{NSreconstruction2}). In the terminology
of \cite{dynlevinexp} this is the \emph{even-order, symmetric and minimal rank
scheme reproducing }$E\left(  \lambda_{0},...,\lambda_{2n-1}\right)  .$ Note
that the Deslauriers-Dubuc scheme is a special case by taking $\lambda
_{0}=...=\lambda_{2n-1}=0,$ reproducing the space of algebraic polynomials
$\Pi_{2n-1}.$

\begin{definition}
\label{DefTotal}For given real numbers $\lambda_{0},...,\lambda_{2n-1}$ let
$a^{\left[  k\right]  },k\in\mathbb{N}_{0}$, be the symbols satisfying
(\ref{NSinterpolation}) and (\ref{NSreconstruction2}). For any natural number
$m\in\mathbb{N}_{0}$ we define a new subdivision scheme $S_{m}^{\Lambda_{2n}}$
defined by the symbols
\[
a^{\left[  k\right]  ,m}\left(  z\right)  =a^{\left[  k+m\right]  }\left(
z\right)  \text{ for }k\in\mathbb{N}_{0}.
\]

\end{definition}

In the rest of the paper we shall use the notations given in Definition
\ref{DefTotal}.

\begin{remark}
It is easy to see that the subdivision scheme $S_{m}^{\Lambda_{2n}}$ is again
an \emph{even-order, symmetric and minimal rank scheme reproducing }$E\left(
\lambda_{0}/2^{m},...,\lambda_{2n-1}/2^{m}\right)  $.
\end{remark}

Many properties of the subdivision scheme $S_{m}^{\Lambda_{2n}}$ for
exponential polynomials can be derived from their polynomial counterpart, the
Deslauriers-Dubuc scheme. The key to these results depend on the following
observation in \cite[Theorem 2.7]{dynlevinexp}.

\begin{proposition}
\label{Pequivalent} The subdivision scheme $S_{m}^{\Lambda_{2n}}$ is
exponentially asymptotically equivalent to the $2n$-point Deslauriers-Dubuc
subdivision scheme, i.e. there exists $C>0$ such that
\begin{equation}
\sum_{|j|\leq2n-1}|p_{j}-a_{j}^{[k],m}|\leq C2^{-k}\text{ for all }%
k\in\mathbb{N}_{0}. \label{aAsymptEquivalent}%
\end{equation}

\end{proposition}

Deslauriers and Dubuc showed in \cite{deslaurierDubuc1989} that their scheme
is $C^{0}$-convergent implying the existence of a basic limit function which
will be denoted in the following by $\Phi^{D_{2n}}.$ Furthermore, one can find
sufficient criteria for the $C^{\ell}$-convergence in
\cite{deslaurierDubuc1989}.

According to Theorem 2.10 in \cite{dynlevinexp} we have the following result

\begin{theorem}
Let $\ell\in\mathbb{N}_{0}.$ If the Deslauriers-Dubuc subdivision scheme is
$C^{\ell}$-convergent then the subdivision scheme $S_{m}^{\Lambda_{2n}}$ is
$C^{\ell}$-convergent as well.
\end{theorem}

According to the last theorem the subdivision scheme $S_{m}^{\Lambda_{2n}}$
has a basic limit function which will be denoted in the following by $\Phi
_{m}^{\Lambda_{2n}}$.

A function $f:\mathbb{R}^{d}\rightarrow\mathbb{C}$ is called \emph{Lipschitz
function of order} $\alpha$ (or \emph{H\"{o}lder function of order} $\alpha)$
if there exists a number $L>0$ such that for all $x,y\in\mathbb{R}^{d}$%
\[
\left\vert f\left(  x\right)  -f\left(  y\right)  \right\vert \leq L\left\vert
x-y\right\vert ^{\alpha}.
\]
The set of all Lipschitz functions of order $\alpha$ is denoted by
$Lip(\alpha).$ Functions in $Lip\left(  \alpha\right)  $ can be characterized
via Fourier transform and we just recall Lemma 7.1 in \cite[p. 56]%
{deslaurierDubuc1989}:

\begin{lemma}
\label{Lem1}Let $f:\mathbb{R}\rightarrow\mathbb{C}$ be an integrable function
whose Fourier transform is $g\left(  \xi\right)  .$ We assume that $\left|
\xi\right|  ^{\ell+\alpha}g\left(  \xi\right)  $ is integrable where $\ell
\in\mathbb{N}_{0}$ and $\alpha\in\left[  0,1\right]  .$ If $\alpha=0$ then $f$
is $\ell$ times continuously differentiable. If $\alpha\neq0$, then
$f^{\left(  \ell\right)  }$ is a Lipschitz function of order $\alpha.$
\end{lemma}

\begin{theorem}
\label{AsymFourier1} Let $\alpha\in\left[  0,1\right)  $ and $\ell
\in\mathbb{N}_{0}$ and assume that the basic limit function $\Phi^{D_{2n}}$ of
the $2n$-point Deslauriers-Dubuc scheme satisfies for some $\varepsilon>0$ and
$C>0$ the inequality
\begin{equation}
\left\vert \widehat{\Phi^{D_{2n}}}(\omega)\right\vert \leq C\left(
|\omega|+1\right)  ^{-\ell-1-\alpha-\varepsilon} \label{ltalpha1}%
\end{equation}
for all $\omega\in\mathbb{R}.$ Then the basic limit function $\Phi
_{m}^{\Lambda_{2n}}$ of the scheme $S_{m}^{\Lambda_{2n}}$ has its $\ell-$th
derivative in $Lip(\alpha)$.
\end{theorem}

\begin{proof}
We begin with some general remarks: let us define $z_{j}^{[k]}=\exp\left(
-2^{-(k+1)}\lambda_{j}\right)  .$ Using (\ref{NSreconstruction2}) we can
write
\begin{equation}
a^{[k]}(z)=\left(  \prod_{j=1}^{2n}{\frac{z+z_{j}^{[k]}}{2}}\right)
\ b^{[k]}(z).\label{abk}%
\end{equation}
Let $D_{2n}\left(  z\right)  =\left(  \frac{1+z}{2}\right)  ^{2n}b_{D_{2n}%
}(z)$ be the symbol of the Deslauriers-Dubuc scheme. It was shown in
\cite[Formula (2.32)]{dynlevinexp} that $b^{[k]}$ is exponentially
asymptotically equivalent to $b_{D_{2n}},$ so there exists a constant $B>0$
such that
\begin{equation}
\left\vert b^{[k]}\left(  e^{i\omega}\right)  -b_{D_{2n}}\left(  e^{i\omega
}\right)  \right\vert \leq B\cdot2^{-k}\label{bk}%
\end{equation}
for all $\omega\in\mathbb{R}$ and for all $k\in\mathbb{N}_{0}.$ As an
intermediate step we consider the non-stationary scheme $S_{m}^{c}$ defined by
the symbols
\begin{equation}
c_{m}^{[k]}(z)={\left(  {\frac{{z+1}}{2}}\right)  }^{2n}b^{[m+k]}%
(z).\label{cbk}%
\end{equation}
According to the proof of Theorem 2.10 in \cite{dynlevinexp} the scheme
$S_{m}^{c}$ has a basic limit function denoted by $\Phi_{m}^{c}$. Let
$\Phi_{m}^{\Lambda_{2n}}$ be the basic limit function of $S_{m}^{\Lambda_{2n}%
}.$ By Proposition \ref{PropProd}
\begin{equation}
\widehat{\Phi_{m}^{\Lambda_{2n}}}(\omega)=\prod_{k=1}^{\infty}{\frac{1}{2}%
}a^{[m+k-1]}(e^{i\omega2^{-k}}).\label{Fprod}%
\end{equation}
Similarly, $\widehat{\Phi^{D_{2n}}}(\omega)=\prod_{k=1}^{\infty}{\frac{1}{2}%
}a(e^{i\omega2^{-k}})$ and we see that
\begin{align}
{\frac{\widehat{\Phi_{m}^{c}}(\omega)}{\widehat{\Phi^{D_{2n}}}(\omega)}} &
=\prod_{k=1}^{\infty}{\frac{b^{[m+k-1]}\left(  e^{i\omega2^{-k}}\right)
}{b_{D_{2n}}\left(  e^{i\omega2^{-k}}\right)  }}\label{ratio}\\
&  =\prod_{k=1}^{\infty}\left(  1+{\frac{b^{[m+k-1]}\left(  e^{i\omega2^{-k}%
}\right)  -b_{D_{2n}}\left(  e^{i\omega2^{-k}}\right)  }{b_{D_{2n}}\left(
e^{i\omega2^{-k}}\right)  }}\right)  .\label{ratio2}%
\end{align}
It is well known \cite{deslaurierDubuc1989} that the trigonometric polynomial
$b_{D_{2n}}\left(  e^{-i\omega2^{-k}}\right)  $ does not vanish on the unit
circle, hence the denominator in (\ref{ratio2}) satisfies
\begin{equation}
\left\vert b_{D_{2n}}\left(  e^{i\omega2^{-k}}\right)  \right\vert \geq
\delta>0\label{bbk}%
\end{equation}
for some $\delta>0.$ Using (\ref{bbk}) and (\ref{bk}) it is straightforward to
prove that the infinite product in (\ref{ratio}) is uniformly bounded for
$\omega\in\mathbb{R}$, and we obtain
\[
\left\vert {\frac{\widehat{\Phi_{m}^{c}}(\omega)}{\widehat{\Phi^{D_{2n}}%
}(\omega)}}\right\vert \leq M_{m}\ .
\]
Recalling that each factor ${\frac{1+z}{2}}$ induces convolution with a
$B-$spline of order $0$ (denoted as usually by $B_{0}$) the above inequality
can be written as
\begin{align}
\left\vert \widehat{\Phi_{m}^{c}}(\omega)\right\vert  &  =\left(  \frac
{\sin{\frac{\omega}{2}}}{{\frac{\omega}{2}}}\right)  ^{2n}\prod_{k=1}^{\infty
}\left\vert b^{[m+k-1]}\left(  e^{i\omega2^{-k}}\right)  \right\vert \leq
M_{m}\left\vert \widehat{\Phi^{D_{2n}}}(\omega)\right\vert \label{sinb}\\
&  \leq M_{n}C(|\omega|+1)^{-\ell-1-\alpha-\varepsilon}\ .
\end{align}
Comparing (\ref{abk}) and (\ref{cbk}) we see that in order to prove the
theorem we need to replace each factor $\frac{\sin{\frac{\omega}{2}}}%
{{\frac{\omega}{2}}}$ by the Fourier transform of the basic limit function
generated by the subdivision scheme with symbols $\left\{  \frac{z+z_{j}%
^{[k]}}{2}\right\}  .$ This should be done with care since $\frac{\sin
{\frac{\omega}{2}}}{{\frac{\omega}{2}}}$ vanishes at infinite number of points
on $\mathbb{R}$. We employ another observation from \cite{dynLevinJMAA95},
that the basic limit function $\Phi_{0}$ of the scheme with symbols $\left\{
\frac{z+z_{j}^{[k]}}{2}\right\}  $ is an exponential $B-$spline of order 0,
\[
B_{0}^{j}\left(  \omega\right)  :=\left\{
\begin{array}
[c]{c}%
e^{\lambda_{j}x}\qquad\text{for }x\in\left[  0,1\right]  \\
0\qquad\qquad\text{otherwise}%
\end{array}
\right.  .
\]

Adding such factors, $\left\{  \frac{z+z_{j}^{[k]}}{2}\right\}  _{j=1}^{2n}$
to the symbols $\{c^{[k]}\}$ results in repeated convolutions of
$\widehat{\Phi_{m}^{c}}$ with $\{B_{0}^{j}\}_{j=1}^{2n}$. Since $\widehat
{B_{0}^{j}}(x)$ decays as ${\frac{1}{|x|}}$ for $|\omega|\longrightarrow
\infty$, each convolution adds one power to the decay power $\alpha$ in
(\ref{sinb}). Hence, after $2n$ convolutions we obtain
\begin{align}
\widehat{\Psi_{m}}(\omega) &  :=\left(  \frac{\sin{\frac{\omega}{2}}}%
{{\frac{\omega}{2}}}\right)  ^{2n}\prod_{k=1}^{\infty}a^{[m+k-1]}b\left(
e^{-i\omega2^{-k}}\right)  \label{sinb1}\\
&  =\left(  \frac{\sin{\frac{\omega}{2}}}{{\frac{\omega}{2}}}\right)
^{2n}\widehat{\Phi_{m}^{\Lambda_{2n}}}\left(  \omega\right)  =O\left(  \left(
|\omega|+1\right)  ^{-\ell-1-\alpha-2n}\right)  \ .\nonumber
\end{align}
It follows by Lemma \ref{Lem1} that the $\left(  2n+\ell\right)  -$th
derivative of ${\Psi}_{m}$ is $Lip(\alpha)$. Now we are ready to return to
$\widehat{\Phi_{m}^{\Lambda_{2n}}}$ by removing the factors $\left(
\frac{\sin{\frac{\omega}{2}}}{{\frac{\omega}{2}}}\right)  ^{2n}$ from
$\widehat{\Psi_{m}}(\omega)$. We need to show that each factor removed implies
that the order of the derivative which is in $Lip(\alpha)$ is reduced by one.
To show this we consider $g=B_{0}\ast f$, where $f$ is a function of compact
support and $g^{\prime}$ is in $Lip(\alpha)$. It follows that
\[
g^{\prime}(s)=f(s)-f(s-1).
\]
Summing the above relations over all numbers $s\in\{t+j\}_{j=0}^{N}$, for
large enough $N$, we obtain
\[
f(t)=\sum_{j=0}^{N}g^{\prime}(t+j)\ ,
\]
implying that $f$ is $Lip(\alpha)$. Since $\hat{g}(\omega)={\frac{\sin
{\frac{\omega}{2}}}{{\frac{\omega}{2}}}}\hat{f}(\omega)$, we have just shown
that the consequence of removing a factor $\frac{\sin{\frac{\omega}{2}}%
}{{\frac{\omega}{2}}}$ is that the order of the derivative which is in
$Lip(\alpha)$ is reduced by one. Removing $2n$ such factors yields the desired
result, namely, that the $\ell-$th derivative of $\{\Phi_{m}^{\Lambda_{2n}}\}$
is in $Lip(\alpha)$.
\end{proof}

\begin{remark}
In \cite{DynKounchevLevinRender2}, \cite{kounchevKalagTsvet2010},
\cite{kounchevKalag2010} motivated by the multivariate polyharmonic
subdivision and wavelets on parallel hyperplanes, an explicit expression for
the polynomial $b^{\left[  k\right]  }\left(  z\right)  $ is found, for  the
case of frequencies given by $\lambda_{j}=\xi,$ $j=0,1,...,n-1,$ for some real
$\xi\geq0,$ and $\lambda_{j}=-\xi,$ for $j=n,...,2n.$ The classical
Deslaurier-Dubuc case corresponds to $\xi=0.$ 
\end{remark}

\section{Non-stationary multiresolutional analysis}

The concept of a multi-resolutional analysis, introduced by S. Mallat and Y.
Meyer, is an effective tool to construct wavelets in a simple way from a given
scaling function $\varphi$, see e.g. \cite{blatter}, \cite{daubechies}.
\emph{Non-stationary} \emph{multiresolution analysis} was introduced in
\cite{deBoorDeVoreRon} by C. de Boor, R. DeVore and A. Ron. For convenience of
the reader we recall here the definition for the univariate case:

\begin{definition}
A (non-stationary) multiresolution analysis consists of a sequence of closed
subspaces $V_{m},m\in\mathbb{Z},$ in $L^{2}\left(  \mathbb{R}\right)  $ satisfying

(i) $V_{m}\subset V_{m+1}$ for all $m\in\mathbb{Z},$

(ii) the intersection $\cap_{m\in\mathbb{Z}}V_{m}$ is the trivial subspace
$\left\{  0\right\}  ,$

(iii) the union $\cup_{m\in\mathbb{Z}}V_{m}$ is dense in $L^{2}\left(
\mathbb{R}\right)  $,

(iv) for each $m\in\mathbb{Z}$ there exists a function $\varphi_{m}\in V_{m}$
such that the family of functions $\left\{  \varphi_{m}\left(  2^{m}%
t-k\right)  :k\in\mathbb{Z}\right\}  $ form a Riesz basis of $V_{m}.$
\end{definition}

The function $\varphi_{m}$ in condition (iv) is called a \emph{scaling
function} for $V_{m}.$ The requirement (iv) means that for each $f\in V_{m}$
there exists a unique sequence $\left(  c_{k}\right)  _{k\in\mathbb{Z}}$ in
$l^{2}\left(  \mathbb{Z}\right)  $ (i.e. that $\sum_{k=-\infty}^{\infty
}\left\vert c_{k}\right\vert ^{2}<\infty)$ such that
\[
f\left(  t\right)  =\sum_{k=-\infty}^{\infty}c_{k}\varphi_{m}\left(
2^{m}t-k\right)
\]
with convergence in $L^{2}\left(  \mathbb{R}\right)  $ and
\[
A_{m}\sum_{k=-\infty}^{\infty}\left\vert c_{k}\right\vert ^{2}\leq\left\Vert
\sum_{k=-\infty}^{\infty}c_{k}\varphi_{m}\left(  2^{m}t-k\right)  \right\Vert
^{2}\leq B_{m}\sum_{k=-\infty}^{\infty}\left\vert c_{k}\right\vert ^{2}%
\]
for all $\left(  c_{k}\right)  _{k\in\mathbb{Z}}$ in $l^{2}\left(
\mathbb{Z}\right)  $ with $0<A_{m}\leq B_{m}<\infty$ constants independent of
$f\in V_{m}$.

Using (i) one can define the \emph{wavelet space} $W_{m}$ as the unique
subspace such that $V_{m}\oplus W_{m}=V_{m+1}$ for $m\in\mathbb{Z}$ and
$W_{m}$ is orthogonal to $V_{m}.$ It easy to see that this implies that $W_{k}
$ and $W_{m}$ are orthogonal subspaces for $k\neq m.$ Conditions (ii) and
(iii) imply that
\[
L^{2}\left(  \mathbb{R}\right)  =\oplus_{m\in\mathbb{Z}}W_{m}.
\]

We refer to \cite{deBoorDeVoreRon} for an extensive discussion on the
construction of so-called pre-wavelets for a nonstationary multiresolutional
analysis, see also \cite{Vonesch}. Important examples of nonstationary
multiresolutions occur in the context of cardinal exponential-splines wavelets
which generalizes the work of C.K. Chui and J.Z. Wang about cardinal spline
wavelets in \cite{ChWa92}, see \cite{Chui1}, \cite{Chui2}. The interested
reader may consult \cite{Micc76}, \cite{Schu81} for the theory of exponential
splines and \cite{deBoorDeVoreRon}, \cite{okbook}, \cite{kounchevRender},
\cite{kounchevrenderpams}, \cite{LySc93} for the construction of wavelets in
this context. In passing we mention that the results in the last cited papers
have been rediscovered in \cite{UnBl05}.

\begin{definition}
A multiresolutional analysis is called stationary if in condition (iv) the
scaling function $\varphi_{m}\in V$ is the same for all $m\in\mathbb{Z}.$
\end{definition}

In this paper we shall be concerned only with orthonormal non-stationary MRA:

\begin{definition}
A multiresolutional analysis is called orthonormal if in condition (iv) the
functions $t\longmapsto2^{m/2}\varphi_{m}\left(  2^{m}t-k\right)  $ for
$k\in\mathbb{Z}$ are an orthonormal basis of $V_{m}.$
\end{definition}

Let us recall that an \emph{orthonormal wavelet} $\psi$ is a function in
$L^{2}\left(  \mathbb{R}\right)  $ such that the system $\psi_{m,k}\left(
x\right)  =2^{m/2}\psi\left(  2^{m}x-k\right)  $ with $m,k\in\mathbb{Z}$ is an
orthonormal basis of $L^{2}\left(  \mathbb{R}\right)  .$ In the context of
nonstationary wavelet analysis one wants to find a sequence of functions
$\psi_{m}\in L^{2}\left(  \mathbb{R}\right)  ,m\in\mathbb{Z},$ such that
\[
\psi_{m,k}\left(  x\right)  =2^{m/2}\psi_{m}\left(  2^{m}x-k\right)
\]
with $m,k\in\mathbb{Z}$ is an orthonormal basis of $L^{2}\left(
\mathbb{R}\right)  .$

\section{Daubechies type wavelets}

Daubechies wavelets $\psi$ are orthonormal wavelets with compact support and
certain degree of smoothness. The construction of Daubechies wavelets is often
presented in the following way (see e.g. \cite{blatter}, \cite{HeWe96},
\cite{LMR98}) : using the concept of an orthonormal MRA it suffices to
construct a suitable scaling function $\varphi.$ Elementary considerations
show that the Fourier transform $\widehat{\varphi}$ of the scaling function
$\varphi$ should be of the form
\[
\widehat{\varphi}\left(  \omega\right)  =\prod\limits_{k=1}^{\infty}m\left(
2^{-k}\omega\right)
\]
where $m\left(  \omega\right)  $ is a trigonometric polynomial with real
coefficients and $m\left(  0\right)  =1$ satisfying the equation%
\begin{equation}
\left\vert m\left(  \omega\right)  \right\vert ^{2}+\left\vert m\left(
\omega+\pi\right)  \right\vert ^{2}=1. \label{eqCQF}%
\end{equation}
This leads to the question which \emph{non-negative} trigonometric polynomials
$q\left(  \omega\right)  $ satisfy an equation of the type
\begin{equation}
q\left(  \omega\right)  +q\left(  \omega+\pi\right)  =1\text{ and }q\left(
0\right)  =1. \label{eqCQF2}%
\end{equation}
There are many explicit solutions of (\ref{eqCQF2}). For example, if $n$ is a
natural number then the trigonometric polynomial
\[
q_{n}\left(  \omega\right)  =1-c_{n}\int\limits_{0}^{\omega}\left(  \sin
t\right)  ^{2n-1}dt
\]
with $c_{n}:=\int\limits_{0}^{\pi}\sin^{2n-1}tdt$ satisfies equation
(\ref{eqCQF2}). By the Fej\'{e}r-Riesz lemma one can find a (non-unique)
trigonometric polynomial $m\left(  \omega\right)  $ such that
\begin{equation}
q_{n}\left(  \omega\right)  =\left\vert m\left(  \omega\right)  \right\vert
^{2}. \label{eqDaub1}%
\end{equation}
We call a Laurent polynomial $m\left(  \omega\right)  $ with real coefficients
and $m\left(  0\right)  =1$ satisfying (\ref{eqDaub1}) a \emph{Daubechies
filter of order }$n.$ The \emph{Daubechies scaling function }$\varphi^{m}%
$\emph{\ for the Daubechies filter} $m\left(  \omega\right)  $ is then defined
by
\[
\widehat{\varphi^{m}}\left(  \omega\right)  =\prod\limits_{k=1}^{\infty
}m\left(  2^{-k}\omega\right)  .
\]
This procedure is elegant but the construction of the trigonometric polynomial
$q_{n}$ in the above approach seems to be rather miraculous. Let us emphasize
that I. Daubechies has shown much more (see e.g. \cite[p. 210]{daubechies} or
\cite{walnut}): the regularity of the wavelet and the scaling function imply
that the symbol $m\left(  \omega\right)  $ must contain a factor $\left(
1+e^{i\omega}\right)  ^{n}/2^{n}.$ Hence $q_{n}\left(  \omega\right)  $ is of
the form
\[
q_{n}\left(  \omega\right)  =\frac{\left(  1+e^{i\omega}\right)  ^{2n}}{2^{n}%
}F_{2n-1}\left(  \omega\right)
\]
where $F_{2n-1}\left(  \omega\right)  $ is a suitable trigonometric polynomial
with real coefficients which can be determined by Bezout's theorem from
(\ref{eqCQF2}). Indeed, it follows from these considerations that
\[
q_{n}\left(  \omega\right)  =D_{2n}\left(  e^{i\omega}\right)
\]
where $D_{2n}$ is the symbol of the Deslauries-Dubuc subdivision scheme, a
fact which is already mentioned by Daubechies in her book \cite[Section
6.5]{daubechies} giving credit to this observation to M.J. Shensa in
\cite{Shen91}, see \cite[p. 210]{daubechies}. \emph{Hence the
Deslauriers-Dubuc scheme leads in a very natural and direct way to the
construction of the Daubechies scaling function and therefore, by MRA-methods,
to Daubechies wavelets.}

In this section we want to use this concept in order to define Daubechies type
wavelets for exponential polynomials. In this setting we have some additional
freedom which is interesting for applications: we may choose real numbers
$\lambda_{0},...,\lambda_{n-1}$ and we shall construct \emph{Daubechies type
wavelets reconstructing the space} $E\left(  \lambda_{0},...,\lambda
_{n-1}\right)  .$ In the case of Daubechies wavelets this corresponds to the
fact that the Daubechies wavelet reproduces polynomials of degree $\leq n-1.$
$.$

We shall write shortly $\Lambda_{0}=\left(  \lambda_{0},...,\lambda
_{n-1}\right)  $ and define $\lambda_{n+j}:=-\lambda_{j}$ for $j=0,...,n-1$.
We consider now the subdivision scheme based on interpolation in $E\left(
\lambda_{0},...,\lambda_{2n-1}\right)  .$ According to Definition
\ref{DefTotal} the subdivision scheme $S_{0}^{\Lambda_{2n}}$ has the symbols
\begin{equation}
a^{\left[  k\right]  }\left(  z\right)  =\prod_{j=0}^{2n-1}{\frac
{z+z_{j}^{[k]}}{2}}\ b^{[k]}(z)\text{ with }z_{j}^{[k]}:=\exp\left(
-2^{-(k+1)}\lambda_{j}\right)  \label{abk2}%
\end{equation}
for $k\in\mathbb{N}_{0}$ and $j=0,...,2n-1.$ Crucial is the result by
Micchelli in \cite{micchelliWAnonstationary} (Proposition 5.1), proving that
the symbols $a^{[k]}$ satisfy
\begin{equation}
a^{[k]}(z)\geq0\ ,\ \text{for }\ |z|=1\ , \label{akzGE0}%
\end{equation}
with equality possible only if $z=-1$. By the Fejer-Riesz lemma we can find
for each level $k$ a \textquotedblright square root\textquotedblright\ Laurent
polynomial $M^{[k]}(z)$ with \emph{real} coefficients, satisfying
\begin{equation}
a^{[k]}(e^{i\omega})={\frac{1}{2}}|M^{[k]}(e^{i\omega})|^{2}\text{ and
}M^{\left[  k\right]  }\left(  1\right)  >0. \label{akMk}%
\end{equation}
Note that this implies that
\[
a^{[k]}(z)={\frac{1}{2}}M^{[k]}(z)M^{[k]}(\frac{1}{z})
\]
for all complex $z\neq0.$ Again, there are many Laurent polynomials
$M^{\left[  k\right]  }\left(  z\right)  $ which satisfy (\ref{akMk}) and all
possible choices can be described through suitable subsets of the zero-set of
$a^{[k]}(z).$ First we choose the roots $z=-\exp\left(  -\lambda_{j}%
/2^{k+1}\right)  $ for $j=0,...,n-1$, in order to obtain stepwise reproduction
of the space $E\left(  \lambda_{0},...,\lambda_{n-1}\right)  ,$ see
Proposition \ref{PropRe} below. Further, we have to choose another $n-1$ roots
of the factor $b^{[k]}$ in (\ref{abk2}). Since $b^{[k]}$ is symmetric, its
$2n-2$ roots come in inverse pairs, say $z_{i}$ and $z_{i}^{-1}$, and as well
complex conjugates $\overline{z_{i}}$ and $\overline{z_{i}}^{-1}$ if $z_{i}$
is not real, for $i$ in an index set $I_{n-1}.$ We choose either the set
$\left\{  z_{i},\overline{z_{i}}\right\}  $ or the set $\left\{  z_{i}%
^{-1},\overline{z_{i}}^{-1}\right\}  $ for each $i\in I_{n-1},$ leading to a
Laurent polynomial with real coefficients which still has to be normalized so
that $M^{\left[  k\right]  }\left(  1\right)  =\sqrt{2a^{[k]}(1)}>0$. We shall
call a sequence of filters $M^{\left[  k\right]  }\left(  z\right)
,\ k\in\mathbf{N}_{0},$ chosen in this way a \emph{non-stationary Daubechies
type subdivision scheme of order} $n.$

Since $M^{\left[  k\right]  }\left(  1\right)  $\ is positive it follows that
$1\leq\frac{1}{2}M^{\left[  k\right]  }\left(  1\right)  +1$ and
$a^{[k]}(1)={\frac{1}{2}}M^{[k]}(1)^{2}$ we infer that
\[
\left\vert \frac{1}{2}M^{\left[  k\right]  }\left(  1\right)  -1\right\vert
\leq\left\vert \frac{1}{2}M^{\left[  k\right]  }\left(  1\right)
-1\right\vert \left\vert \frac{1}{2}M^{\left[  k\right]  }\left(  1\right)
+1\right\vert =\left\vert \frac{1}{2}a^{\left[  k\right]  }\left(  1\right)
-1\right\vert .
\]
Since $a^{\left[  k\right]  }\left(  z\right)  $ is exponentially
asymptotically equivalent to the Deslauriers-Dubuc scheme and $\frac{1}%
{2}D_{2n}\left(  1\right)  =1$ we infer that there exists $C>0$ such that for
all $k\in\mathbb{N}_{0}$
\begin{equation}
\left\vert \frac{1}{2}M^{\left[  k\right]  }\left(  1\right)  -1\right\vert
\leq C\cdot2^{-k} \label{eqMM}%
\end{equation}

At first we notice the following result:

\begin{proposition}
\label{PropRe}Let $\lambda_{0},...,\lambda_{n-1}$ be real numbers. Then there
exists $k_{0}\in\mathbb{N}_{0}$ such that the Daubechies type subdivision
scheme reproduces stepwise functions in $E\left(  \lambda_{0},....,\lambda
_{n-1}\right)  $ for all levels $k\geq k_{0}.$
\end{proposition}

\begin{proof}
Let $z_{j}^{\left[  k\right]  }=\exp\left(  -\lambda_{j}/2^{k+1}\right)  .$ By
construction $M^{\left[  k\right]  }\left(  z\right)  $ has a zero at
$-z_{j}^{\left[  k\right]  }$ of multiplicity $\mu_{j}$, the number of times
$\lambda_{j}$ occurs in $\left(  \lambda_{0},...,\lambda_{n-1}\right)  ,$
hence%
\begin{equation}
\frac{d^{s}}{dx^{s}}M^{\left[  k\right]  }\left(  -z_{j}^{\left[  k\right]
}\right)  =0\text{ for }s=0,...,\mu_{j}-1\text{ and }j=1,...,n-1.
\label{eqlast}%
\end{equation}
By (\ref{akMk}) and the fact that $\left\{  a^{\left[  k\right]  }\right\}  $
reproduces stepwise functions in $E\left(  \lambda_{0},...,\lambda
_{2n-1}\right)  $ we conclude that
\[
{\frac{1}{2}}|M^{[k]}(z_{j}^{\left[  k\right]  })|^{2}=a^{[k]}(z_{j}^{\left[
k\right]  })=2.
\]
Since $z_{j}^{\left[  k\right]  }$ is real and $M^{\left[  k\right]  }\left(
z\right)  $ has real coefficients it follows that $M^{[k]}(z_{j}^{\left[
k\right]  })$ is real so $M^{[k]}(z_{j}^{\left[  k\right]  })=2$ or $-2.$
Since $z_{j}^{\left[  k\right]  }$ converges to $1$ for $k\rightarrow\infty$
and $M^{[k]}(z_{j}^{\left[  k\right]  })$ converges to $M^{D_{2n}}(1)>0$ there
exists $k_{0}\in\mathbf{N}_{0}$ such that $M^{[k]}(z_{j}^{\left[  k\right]
})>0$ for all $k\geq k_{0}$ and $j=1,...n-1.$ Hence $M^{[k]}(z_{j}^{\left[
k\right]  })=2$ for all $k\geq k_{0}.$ From (\ref{akMk}) we infer that for
real $x$
\[
\frac{d^{s}}{dx^{s}}a^{\left[  k\right]  }\left(  x\right)  =\sum_{r=0}%
^{s}\binom{s}{r}\frac{d^{r}}{dx^{r}}M^{\left[  k\right]  }\left(  x\right)
\cdot\frac{d^{s-r}}{dx^{s-r}}M^{\left[  k\right]  }\left(  x\right)  .
\]
For $s=1$ this means that $0=M^{\left[  k\right]  }\left(  z_{j}^{\left[
k\right]  }\right)  \frac{d}{dx}M^{\left[  k\right]  }\left(  z_{j}^{\left[
k\right]  }\right)  +\frac{d}{dx}M^{\left[  k\right]  }\left(  z_{j}^{\left[
k\right]  }\right)  \cdot M^{\left[  k\right]  }\left(  z_{j}^{\left[
k\right]  }\right)  .$ Since $M^{\left[  k\right]  }$ has real coefficients
and $z_{j}^{\left[  k\right]  }$ is real we conclude that $\frac{d}%
{dx}M^{\left[  k\right]  }\left(  z_{j}^{\left[  k\right]  }\right)  =0.$
Inductively we obtain that $\frac{d^{s}}{dx^{s}}M^{\left[  k\right]  }\left(
z_{j}^{\left[  k\right]  }\right)  =0$ for $s=1,...,\mu_{j}-1.$
\end{proof}

\begin{proposition}
The product $\prod_{k=1}^{\infty}\frac{1}{2}M^{\left[  k-1\right]  }\left(
e^{i\frac{\omega}{2^{k}}}\right)  $ converges.
\end{proposition}

\begin{proof}
By construction $M^{\left[  k\right]  }\left(  e^{i\omega}\right)  $ has real
coefficients and $M^{\left[  k\right]  }\left(  1\right)  >0.$ Since
$a^{\left[  k\right]  }\left(  e^{i\omega}\right)  \geq0$ we infer from
(\ref{NSinterpolation}) that $\left\vert a^{\left[  k\right]  }\left(
e^{i\omega}\right)  \right\vert \leq2$, and therefore $\left\vert M_{k}\left(
e^{i\omega}\right)  \right\vert \leq2$. Moreover, it follows from (\ref{eqMM})
that
\[
\sum_{k=1}^{\infty}\left\vert \frac{1}{2}M^{\left[  k\right]  }\left(
1\right)  -1\right\vert \leq C\sum_{k=1}^{\infty}2^{-k}.
\]
Proposition \ref{PropSuff} finishes the proof.
\end{proof}

In order to obtain asymptotic equivalence for the non-stationary Daubechies
subdivision scheme we have to choose the filter $M^{\left[  k\right]  }\left(
z\right)  $ with more care:

\begin{theorem}
\label{ThmMainP}Let $M^{D_{2n}}$ be the Daubechies filter of order $n$ as
defined above and let $\lambda_{0},...,\lambda_{n-1}$ be real numbers. Then
there exists a non-stationary Daubechies type subdivison scheme $\{M^{[k]}%
(z)\}$ which is exponentially asymptotically equivalent to $M^{D_{2n}}$ and
reproduces $e^{\lambda_{j}x}$ stepwise for all $k\geq k_{0}$ and for
$j=0,...,n-1.$
\end{theorem}

\begin{proof}
Note that $M^{D_{2n}}\left(  z\right)  $ has $n$ zeros at $-1$ and $n-1$ other
zeros, say $\alpha_{1},....,\alpha_{n-1}$ which are of course zeros of the
factor $b_{D_{2n}}\left(  z\right)  $ of the Deslauriers-Dubuc symbol
$D_{2n}\left(  z\right)  =\left(  \frac{1+z}{2}\right)  ^{2n}b_{D_{2n}}\left(
z\right)  .$ Using (\ref{eqD2nB}) and Proposition \ref{PropQQQ} we see that
$\alpha_{1},...,\alpha_{n-1}$ are pairwise different and simple zeros of
$D_{2n}\left(  z\right)  .$ Recall that $a^{\left[  k\right]  }\left(
z\right)  $ is exponentially asymptotically equivalent to the symbols
$D_{2n}\left(  z\right)  .$ Then $z^{2n-1}a^{\left[  k\right]  }\left(
z\right)  $ are polynomials and $z^{2n-1}a^{\left[  k\right]  }\left(
z\right)  $ is obviously exponentially asymptotically equivalent to
$z^{2n-1}D_{2n}\left(  z\right)  .$ By Theorem \ref{ThmNec} there exists a
constant $C>0$ and a zero $\alpha_{j}^{\left[  k\right]  }$ of $a^{\left[
k\right]  }\left(  z\right)  $ such that $\left\vert \alpha_{j}^{\left[
k\right]  }-\alpha_{j}\right\vert \leq C2^{-k}$ for all $k\in\mathbb{N}_{0}$
and $j=1,...,n-1.$ Take $k_{0}\in\mathbb{N}_{0}$ large enough so that:\ 

(i) for each $k\geq k_{0}$ the balls $\left\vert z-\alpha_{j}\right\vert \leq
C2^{-k}$ have empty intersection with the unit circle,

(ii) they are pairwise disjoint for $j=1,...n-1,$ and

(iii) they have empty intersection with the $x$-axis if $\alpha_{j}$ is a
non-real zero.

Then for each $k\geq k_{0}$ there is for given $\alpha_{j}$ exactly one zero
$\alpha_{j}^{\left[  k\right]  }$ with%
\begin{equation}
\left\vert \alpha_{j}^{\left[  k\right]  }-\alpha_{j}\right\vert \leq
C2^{-k}\text{ for }j=1,...,n-1, \label{eqalp}%
\end{equation}
leading to a \emph{unique choice} for $M^{\left[  k\right]  }\left(  z\right)
$ for $k\geq k_{0}.$ Further the leading coefficient $c^{\left[  k\right]  }$
of the polynomial $z^{2n-1}M^{\left[  k\right]  }\left(  z\right)  $ is
determined by the equation
\begin{equation}
M^{\left[  k\right]  }\left(  1\right)  =c^{\left[  k\right]  }\prod
_{j=0}^{n-1}\left(  1+z_{j}^{[k]}\right)  {\cdot}\prod_{j=1}^{n-1}\left(
1-\alpha_{j}^{\left[  k\right]  }\right)  . \label{eqlead}%
\end{equation}
Let $c$ be the leading coefficient of $z^{2n-1}D_{2n}\left(  z\right)  $. By
(\ref{eqMM}) and (\ref{eqalp}) and (\ref{eqlead}) it is easy to see that there
exists $D>0$ such that
\[
\left\vert c^{\left[  k\right]  }-c\right\vert \leq D2^{-k}%
\]
for all $k\in\mathbb{N}_{0}.$ By Theorem \ref{ThmSuff} the subdivision scheme
defined by the symbols $z^{2n-1}M^{\left[  k\right]  }\left(  z\right)  $,
$k\in\mathbb{N}_{0},$ is exponentially asymptotically equivalent to the scheme
defined by the symbol $z^{2n-1}M^{D_{2n}}.$ This implies that $M^{\left[
k\right]  }\left(  z\right)  $, $k\in\mathbb{N}_{0},$ is exponentially
asymptotically equivalent to $M^{D_{2n}}.$
\end{proof}

\begin{remark}
Assume that $M^{D_{2n}}$ is the Daubechies filter such that all zeros $\neq-1$
have absolute value bigger than $1$. Then one can define $M^{\left[  k\right]
}\left(  z\right)  $ in the last theorem by the condition that all its
non-trivial zeros have absolute value bigger than $1.$
\end{remark}

It thus follows, from the theory of asymptotically equivalent schemes in
\cite{dynLevinJMAA95}, that the scheme with symbols $\{M^{[k+m]}%
(z),k\in\mathbb{N}_{0}\}$ defines continuous basic limit functions $\left\{
\varphi_{m}^{\Lambda_{0}}\left(  \cdot\right)  \right\}  $. Proposition
\ref{PropProd} shows that
\begin{equation}
\widehat{\varphi_{m}^{\Lambda_{0}}}\left(  \omega\right)  =\prod_{k=1}%
^{\infty}\frac{M^{\left[  m+k-1\right]  }\left(  e^{i\frac{\omega}{2^{k}}%
}\right)  }{2}. \label{eqmprod}%
\end{equation}
In particular, we have
\begin{equation}
\left\vert \widehat{\varphi_{m}^{\Lambda_{0}}}\left(  \omega\right)
\right\vert ^{2}=\widehat{\Phi_{m}^{\Lambda_{2n}}}\left(  \omega\right)
\label{eqauto}%
\end{equation}
where $\Phi_{m}^{\Lambda_{2n}}$ is the basic limit function of $S_{m}%
^{\Lambda_{2n}}$.

The following observation is straightforward and the proof is included for
convenience of the reader:

\begin{proposition}
The functions $2^{m/2}\varphi_{m}^{\Lambda_{0}}\left(  2^{m}\cdot-k\right)
$,$k\in\mathbb{Z}$, are orthonormal.
\end{proposition}

\begin{proof}
Define $D_{k,l}:=2^{m}\int_{-\infty}^{\infty}\varphi_{m}^{\Lambda_{0}}\left(
2^{m}t-k\right)  \overline{\varphi_{m}^{\Lambda_{0}}\left(  2^{m}t-l\right)
}dt.$ A simple transformation of variables and Plancherel's formula gives
\[
D_{k,l}=\int_{-\infty}^{\infty}\varphi_{m}^{\Lambda_{0}}\left(  y-k\right)
\overline{\varphi_{m}^{\Lambda_{0}}\left(  y-l\right)  }dy=\frac{1}{2\pi}%
\int_{-\infty}^{\infty}\widehat{\varphi_{m}^{\Lambda_{0}}}\left(
\omega\right)  \overline{\widehat{\varphi_{m}^{\Lambda_{0}}}\left(
\omega\right)  }e^{-i\left(  k-l\right)  \omega}dt.
\]
Using (\ref{eqauto}) we obtain that
\[
D_{k,l}=\frac{1}{2\pi}\int_{-\infty}^{\infty}\Phi_{m}^{\Lambda_{2n}}\left(
\omega\right)  e^{-i\left(  k-l\right)  \omega}dt=\Phi_{m}^{\Lambda_{2n}%
}\left(  k-l\right)  =\delta_{0,k-l}.
\]

\end{proof}

\begin{definition}
For each $m\in\mathbb{N}_{0}$ we define the linear spaces $V_{m}$ by
\begin{equation}
V_{m}:=\left\{  f\in L^{2}(\mathbb{R})\ |\ f(t)=\sum_{j\in\mathbb{Z}}%
c_{j}\varphi_{m}^{\Lambda_{0}}(2^{m}t-j)\ ,\ \ \sum_{j\in\mathbb{Z}}%
|c_{j}|^{2}<\infty\right\}  \ . \label{Vm}%
\end{equation}

\end{definition}

We remark that we could also define $V_{m}$ for integers $m\in\mathbb{Z}$
since the symbols $a^{\left[  k+m\right]  }\left(  z\right)  $ and the scaling
function $\Phi_{m}^{\Lambda_{2n}}$ and $\varphi_{m}^{\Lambda_{0}}$ could be
defined for all $m\in\mathbb{Z}$.

\begin{proposition}
The spaces $V_{m}$ are nested, i.e. that $V_{m}\subset V_{m+1}$ for all
$m\in\mathbb{N}_{0}.$
\end{proposition}

\begin{proof}
Using the product representation (\ref{eqmprod}) we obtain
\begin{equation}
\widehat{\varphi_{m}^{\Lambda_{0}}}\left(  \omega\right)  =\prod_{k=1}%
^{\infty}\frac{M^{\left[  m+k-1\right]  }\left(  e^{i\frac{\omega}{2^{k}}%
}\right)  }{2}=\frac{M^{\left[  m\right]  }\left(  e^{i\frac{\omega}{2}%
}\right)  }{2}\widehat{\varphi_{m+1}^{\Lambda_{0}}}\left(  \frac{\omega}%
{2}\right)  . \label{eqFscal}%
\end{equation}

Let us write $M^{[m]}(z)=\sum_{j=-n+1}^{n}\mu_{j}^{[m]}z^{j}$ where $\mu
_{j}^{[m]}$ are \emph{real} numbers. Using elementary techniques in Fourier
analysis it is easy to see that the equation (\ref{eqFscal}) is equivalent to
the refinement equation
\begin{equation}
\varphi_{m}^{\Lambda_{0}}(t)=\sum_{j=-n+1}^{n}\mu_{j}^{[m]}\varphi
_{m+1}^{\Lambda_{0}}(2t+j). \label{Mrefinement}%
\end{equation}
Replacing $t$ by $2^{m}t$ in (\ref{Mrefinement}) we obtain that $\varphi
_{m}^{\Lambda_{0}}(2^{m}\cdot)\in V_{m+1}$. Similarly it follows that
$\varphi_{m}^{\Lambda_{0}}(2^{m}t-j)\in V_{m+1}$ for each $j\in\mathbb{Z}.$
\end{proof}

We note that the orthonormality of $\varphi_{m}^{\Lambda_{0}}\left(
t-l\right)  $ implies that%
\begin{equation}
\delta_{0,l}=\int_{-\infty}^{\infty}\varphi_{m}^{\Lambda_{0}}\left(  t\right)
\overline{\varphi_{m}^{\Lambda_{0}}\left(  t-l\right)  }dt=\frac{1}{2}%
\sum_{j\in\mathbb{Z}}\mu_{j}^{[m]}\mu_{j+2l}^{[m]}. \label{eqO5}%
\end{equation}
Indeed, using (\ref{Mrefinement}) in the integral in (\ref{eqO5}) and then
again the orthonormality relations one obtains
\begin{align*}
\delta_{0,l}  &  =\sum_{j,k\in\mathbb{Z}}\mu_{j}^{[m]}\overline{\mu_{k}^{[m]}%
}\int_{-\infty}^{\infty}\varphi_{m+1}^{\Lambda_{0}}(2t+j)\overline
{\varphi_{m+1}^{\Lambda_{0}}(2t-2l+k)}dt\\
&  =\frac{1}{2}\sum_{j,k\in\mathbb{Z}}\mu_{j}^{[m]}\overline{\mu_{k}^{[m]}%
}\int_{-\infty}^{\infty}\varphi_{m+1}^{\Lambda_{0}}(y)\overline{\varphi
_{m+1}^{\Lambda_{0}}(y-j-2l+k)}dy.
\end{align*}

\begin{proposition}
The union of the spaces $V_{m}$ is dense in $L_{2}\left(  \mathbb{R}\right)
.$
\end{proposition}

\begin{proof}
Define $f_{m}\left(  x\right)  =\varphi_{m}^{\Lambda_{0}}\left(
2^{m}x\right)  .$ Then $\widehat{f_{m}}\left(  \omega\right)  =\widehat
{\varphi_{m}^{\Lambda_{0}}}\left(  2^{-m}\omega\right)  .$ By Theorem 4.3 in
\cite{deBoorDeVoreRon} it suffices to show that the set $\Omega,$ defined as
the union of the supports of $\widehat{f_{m}}$ over $m\in\mathbb{N}_{0},$ is
equal to $\mathbb{R}$ minus a set of Lebesgue measure $0.$ Since
$\widehat{\Phi_{m}^{\Lambda_{2n}}}(\omega)=|\widehat{\varphi_{m}^{\Lambda_{0}%
}}(\omega)|^{2}$ we have only to investigate the zeros of $\widehat{\Phi
_{m}^{\Lambda_{2n}}}(\omega).$ By (\ref{eqmprod}) this is an infinite product
of non-negative terms $a^{[m+k]}\left(  e^{i\omega2^{-k}}\right)  $ which may
be zero only if $e^{i\omega2^{-k}}=-1,$ see \cite{micchelliWAnonstationary}.
Hence, the zeros of $\widehat{\Phi_{m}^{\Lambda_{2n}}}$ are of the form
$\omega=2^{k}\left(  \pi+2n\pi\right)  $ for some $n\in\mathbb{Z}$. Hence the
support of $\widehat{\Phi_{m}^{\Lambda_{2n}}}$ is equal to the real line. This
ends the proof.
\end{proof}

\begin{theorem}
\label{AsymFourier2} Let $M^{D_{2n}}(z)$ be a Daubechies filter of order $n$
and assume that $M^{\left[  k\right]  }\left(  z\right)  $ is as in Theorem
\ref{ThmMainP}. Let $\alpha\in\left[  0,1\right)  $ and $\ell\in\mathbb{N}%
_{0}, $ and assume that the scaling function $\varphi^{D_{2n}}$ of Daubechies
defined by the symbol $M^{D_{2n}}(z)$, satisfies for some $\varepsilon>0$ and
$C>0$ the inequality
\[
|\widehat{\varphi^{D_{2n}}}(\omega)|\leq C(|\omega|+1)^{-\ell-1-\alpha
-\varepsilon}%
\]
for all $\omega\in\mathbb{R}.$ Then the scaling function $\varphi_{m}%
^{\Lambda_{0}}$ associated to the subdivision scheme $M^{\left[  k+m\right]
}\left(  z\right)  ,$ $k\in\mathbb{N}_{0},$ have $\ell-$th derivative in
$Lip(\alpha)$.
\end{theorem}

\begin{proof}
Let us define $z_{j}^{[k]}=\exp\left(  -2^{-(k+1)}\lambda_{j}\right)  .$ Then
the symbol $M^{[k]}(z)$ can be written as
\[
M^{[k]}(z)=\left(  \prod_{j=1}^{n}{\frac{z+z_{j}^{[k]}}{2}}\right)
\ B^{[k]}(z).
\]
Similarly we can write for the Daubechies filter
\[
M^{D_{2n}}(z)=\left(  \prod_{j=1}^{n}{\frac{z+1}{2}}\right)  \ B^{D_{2N}}(z).
\]
Since $M^{[k]}(z)$ is exponentially asymptotically equivalent to the
Dauchechies filter $M^{D_{2n}}$ and $B^{[k]}(z)$ has only simple zeros it
follows from Theorems \ref{ThmNec} and \ref{ThmSuff} that $B^{[k]}$ is
exponentially asymptotically equivalent to $B^{D_{2N}}(z),$ so there exists a
constant $C>0$ such that
\[
\left\vert B^{[k]}\left(  e^{i\omega}\right)  -B^{D_{2N}}(z)\left(
e^{i\omega}\right)  \right\vert \leq C\cdot2^{-k}%
\]
Now one can proceed as in Theorem \ref{AsymFourier1}.
\end{proof}

Next we turn to the construction of Daubechies type wavelets. The procedure
will follow the classical pattern in MRA. For convenience of the reader we
shall briefly sketch the construction: Recall that $M^{[m]}(z)=\sum
_{j=-n+1}^{n}\mu_{j}^{[m]}z^{j}$ and we write the refinement equation in
(\ref{Mrefinement}) in the form used in MRA, namely
\[
\varphi_{m}^{\Lambda_{0}}(t)=\sum_{j\in\mathbb{Z}}\mu_{-j}^{[m]}\varphi
_{m+1}^{\Lambda_{0}}(2t-j).
\]
The Daubechies type wavelets $\psi_{m}^{\Lambda_{0}}$ are now defined in the
classical way, namely by \textbf{\ }
\begin{equation}
\psi_{m}^{\Lambda_{0}}(t)=\sum_{j\in\mathbb{Z}}\nu_{j}^{\left[  m\right]
}\varphi_{m+1}^{\Lambda_{0}}(2t-j)\ , \label{psiexplicit}%
\end{equation}
where the coefficients $\{\nu_{k}^{\left[  m\right]  }\}$ are related to those
in (\ref{Mrefinement}) by
\begin{equation}
\nu_{j}^{\left[  m\right]  }=(-1)^{j+1}\mu_{-1+j}^{\left[  m\right]  }\ .
\label{nuk}%
\end{equation}
Then $\psi_{m}^{\Lambda_{0}}$ has compact support since $\varphi
_{m+1}^{\Lambda_{0}}$ has compact support. It is a routine exercise to see
that the system of functions $\left\{  2^{m/2}\psi_{m}^{\Lambda_{0}}\left(
2^{m}t-r\right)  :r\in\mathbb{Z}\right\}  $ is orthonormal: define
\[
D_{r,s}:=2^{m}\int_{-\infty}^{\infty}\psi_{m}^{\Lambda_{0}}\left(
2^{m}t-r\right)  \overline{\psi_{m}^{\Lambda_{0}}\left(  2^{m}t-s\right)
}dt.
\]
The transformation $y=2^{m}t-r$ gives $D_{r,s}=\int_{-\infty}^{\infty}\psi
_{m}^{\Lambda_{0}}\left(  y\right)  \overline{\psi_{m}^{\Lambda_{0}}\left(
y+r-s\right)  }dy$ and therefore
\begin{align*}
D_{r,s}  &  =\sum_{k,l\in\mathbb{Z}}\nu_{k}^{\left[  m\right]  }v_{l}^{\left[
m\right]  }\int_{-\infty}^{\infty}\varphi_{m+1}^{\Lambda_{0}}(2y-k)\overline
{\varphi_{m+1}^{\Lambda_{0}}\left(  2y+2r-2s-l\right)  }dy\\
&  =\frac{1}{2}\sum_{k,l\in\mathbb{Z}}\nu_{k}^{\left[  m\right]  }%
v_{l}^{\left[  m\right]  }\int_{-\infty}^{\infty}\varphi_{m+1}^{\Lambda_{0}%
}(x)\overline{\varphi_{m+1}^{\Lambda_{0}}\left(  x+2r-2s+k-l\right)  }dt\\
&  =\frac{1}{2}\sum_{k\in\mathbb{Z}}\nu_{k}^{\left[  m\right]  }v_{k+2\left(
r-s\right)  }^{\left[  m\right]  }=\frac{1}{2}\sum_{k\in\mathbb{Z}}\mu
_{-1+k}^{\left[  m\right]  }\mu_{-1+k+2s-2r}^{\left[  m\right]  }%
=\delta_{0,r-s}%
\end{align*}
where we have used equation (\ref{eqO5}).

As explained in Section 4 the wavelet space $W_{m}$ is defined as the
orthogonal complement $W_{m}$ of $V_{m}$ in $V_{m+1}$. Next one shows that the
wavelet space $W_{m}$ is equal to
\begin{equation}
\widetilde{W}_{m}:=\left\{  f\in L^{2}(\mathbb{R})\ |\ f(t)=\sum
_{j\in\mathbb{Z}}c_{j}2^{m/2}\psi_{m}^{\Lambda_{0}}(2^{m}t-j)\ ,\ \ \sum
_{j\in\mathbb{Z}}|c_{j}|^{2}<\infty\right\}  . \label{WPsi}%
\end{equation}
We show at first that $\widetilde{W}_{m}$ is orthogonal to $V_{m}\subset
V_{m+1}$ implying that $\widetilde{W}_{m}\subset W_{m}.$ So we look at
\[
C_{r,s}:=2^{m}\int_{-\infty}^{\infty}\varphi_{m}^{\Lambda_{0}}\left(
2^{m}t-r\right)  \overline{\psi_{m}^{\Lambda_{0}}(2^{m}t-s)}dt.
\]
Again $C_{r,s}=\int_{-\infty}^{\infty}\varphi_{m}^{\Lambda_{0}}\left(
y\right)  \overline{\psi_{m}^{\Lambda_{0}}(y+r-s)}dt$ and (\ref{Mrefinement})
and (\ref{psiexplicit}) yield
\begin{align*}
C_{r,s}  &  =\sum_{j,l\in\mathbb{Z}}\mu_{-j}^{[m]}\nu_{l}^{\left[  m\right]
}\int_{-\infty}^{\infty}\varphi_{m+1}^{\Lambda_{0}}(2y-j)\overline
{\varphi_{m+1}^{\Lambda_{0}}(2y+2r-2s-l)}dy\\
&  =\frac{1}{2}\sum_{j,l\in\mathbb{Z}}\mu_{-j}^{[m]}\nu_{l}^{\left[  m\right]
}\int_{-\infty}^{\infty}\varphi_{m+1}^{\Lambda_{0}}(x)\overline{\varphi
_{m+1}^{\Lambda_{0}}(x+j+2r-2s-l)}dx\\
&  =\frac{1}{2}\sum_{j\in\mathbb{Z}}\mu_{-j}^{[m]}\nu_{j+2r-2s}^{\left[
m\right]  }=-\frac{1}{2}\sum_{j\in\mathbb{Z}}\mu_{-j}^{[m]}\left(  -1\right)
^{j}\mu_{-1+j+2r-2s}^{\left[  m\right]  }.
\end{align*}
A simple well known argument shows that the last sum is always zero, see e.g.
\cite[p. 123]{LMR98}. The proof that $\widetilde{W}_{m}\oplus V_{m}$ is equal
to $V_{m+1}$ (implying that $W_{m}=\widetilde{W}_{m})$ follows standard
arguments in MRA and is omitted.

\begin{corollary}
The smoothness of the Daubechies type wavelets $\psi_{m}^{\Lambda_{0}}$ is at
least as that of the classical Daubechies wavelet $\psi$
\end{corollary}

\begin{proof}
This follows immediately from formula (\ref{psiexplicit}) and Proposition
\ref{AsymFourier2}.
\end{proof}

Finally we mention that the concept of reproduction is defined in MRA in the
following way, see e.g. \cite{Vonesch}:

\begin{definition}
A non-stationary multiresolutional analysis $\left(  V_{m}\right)
_{m\in\mathbb{Z}}$ with compactly supported scaling functions $\varphi_{m}$
reproduces a function $f:\mathbb{R}\rightarrow\mathbb{C}$ if for each
$m\in\mathbb{Z}$ there exist complex coefficients $c_{m}$ such that
\begin{equation}
f\left(  x\right)  =\sum_{l\in\mathbb{Z}}c_{m}\varphi_{m}\left(
2^{m}x-l\right)  . \label{eqMRARep}%
\end{equation}

\end{definition}

Note that the sum in (\ref{eqMRARep}) is well defined since $\varphi_{m}$ is
compactly supported. It is proved in \cite{Vonesch} that (\ref{eqlast})
implies that the MRA $\left(  V_{m}\right)  _{m\in\mathbb{Z}}$ of the
Daubechies subdivision scheme reproduces the space $E\left(  \lambda
_{0},...,\lambda_{n-1}\right)  $ and the wavelets $\psi_{m}$ have vanishing
moments in the sense that
\[
\int\psi_{m}^{\Lambda_{0}}\left(  t\right)  e^{\lambda_{j}t}dt=0\qquad
\text{for all }j=0,...n-1.
\]

As we mentioned in the introduction an attempt to construct (non-stationary)
Daubechies type wavelets based on the subdivision scheme for exponential
polynomials in analogy to the case of Deslauriers and Dubuc, was carried out
in \cite{Vonesch}. The authors showed that many results and techniques of
classical multiresolutional analysis, briefly MRA, carry over to
non-stationary MRA (except that all filters and scaling functions are
dependent on the scale of the multiresolution) and they introduced
non-stationary Daubechies-type wavelets reproducing a family of exponentials
$\left\{  e^{\lambda_{j}t}\right\}  _{j=1}^{n}$ \emph{under the assumption}
that the filters have sufficiently long support. This strong assumption has
been made in order to guarantee that the symbols of the nonstationary scheme
are non-negative. However, as we mentioned in the introduction, a main result
proven in \cite{micchelliWAnonstationary} shows that the symbols are
non-negative under the simple assumption that the exponents $\lambda_{j}$ for
$j=1,...,n$ are real. As a consequence, the question about  the order of
regularity of the new Daubechies type wavelets, is not addressed in
\cite{Vonesch}.

\end{document}